\documentclass[12pt,reqno,hidelinks]{amsart}
\usepackage{amssymb}
\usepackage{amsmath, amssymb, verbatim, url, mathrsfs} 
\usepackage{stmaryrd}
\usepackage{dutchcal}
\usepackage{tabularx}
\usepackage{mathtools}
\usepackage{verbatim}
\usepackage{amsthm}
\usepackage{framed}
\usepackage{cite}
\usepackage{wasysym}
\usepackage{upgreek}
\usepackage{color}
\usepackage[dvipsnames]{xcolor}
\usepackage{array}
\usepackage{tensor}
\usepackage{accents} 
\usepackage[T1]{fontenc}
\usepackage{dsfont}
\usepackage[colorlinks,linkcolor=blue,citecolor=blue]{hyperref}
\usepackage{enumerate}
\usepackage{enumitem}
\usepackage[normalem]{ulem}
\usepackage{longtable}
\usepackage{mathtools}
\usepackage{subcaption}  
\usepackage{ulem}
\usepackage{float}

\numberwithin{equation}{section}

\theoremstyle{definition} 

\newtheorem{proposition}{Proposition}[section]
\newtheorem{lemma}[proposition]{Lemma}
\newtheorem{corollary}{Corollary}[section]
\newtheorem{theorem}{Theorem}[section]

\newtheorem{remark}{Remark}[section]

\newtheorem*{theorem*}{Theorem}
\newtheorem*{mquestion*}{Main Question}
\newtheorem*{claim*}{Claim}

\newtheorem*{intuition*}{Intuition}

\newcommand{\vertiiii}[1]{{\left\vert\kern-0.25ex\left\vert\kern-0.25ex\left\vert\kern-0.25ex\left\vert #1 \right\vert\kern-0.25ex\right\vert\kern-0.25ex\right\vert\kern-0.25ex\right\vert}}
\newcommand{\vertiii}[1]{{\left\vert\kern-0.25ex\left\vert\kern-0.25ex\left\vert #1 \right\vert\kern-0.25ex\right\vert\kern-0.25ex\right\vert}}

\newcommand{\Rbb}{\mathbb{R}}

\newcommand{\AND}{{\quad\text{and}\quad}}

\newcommand{\p}[1]{
	\begin{pmatrix}
		#1
	\end{pmatrix}
}

\newcommand{\be}{\begin{equation}}
\newcommand{\ee}{\end{equation}}

\allowdisplaybreaks

\begin{document}

	\title[Nonlinear Singularity-free cosmological solutions]{Existence and bounds of nonlinear singularity-free cosmological solutions in a string-inspired gravity}
	
	\author{Chihang He and Chao Liu}

	\address[Chihang He]{School of Mathematics and Statistics, Huazhong University of Science and Technology, Wuhan 430074, Hubei Province, China.}
	\email{hechihang@foxmail.com}

	\address[Chao Liu]{Center for Mathematical Sciences and School of Mathematics and Statistics, Huazhong University of Science and Technology, Wuhan 430074, Hubei Province, China.}
	\email{chao.liu.math@foxmail.com}

	\begin{abstract}  
		We provide a rigorous proof for the existence of  homogeneous,  isotropic and globally singularity-free cosmological solutions in Einstein-dilaton-Gauss-Bonnet (EdGB) gravity with exponential coupling. While numerical studies suggested such solutions exist, a formal proof remained elusive.  By employing a novel ``power identity method'' and overcoming significant challenges posed by the strong nonlinearities of the exponential coupling, which are not present in the quadratic coupling analyzed in our companion paper \cite{he2025proofssingularityfreesolutionsscalarization},  we establish a FLRW solution valid for all time $t\in(-\infty,+\infty)$, where the Hubble parameter remains positive and vanishes asymptotically, while the scalar field evolves monotonically. This result align with numerical simulations and offer a firm mathematical foundation for singularity-free cosmology in a string-inspired setting.

		\vspace{2mm}
		
		
	\end{abstract}

	\maketitle

	\section{Introduction}
	
	In recent years, numerous modified gravity theories have been proposed to address unresolved problems in cosmology, such as the nature of dark energy and dark matter. Among them, the Einstein–dilaton–Gauss–Bonnet (EdGB) gravity is a string-inspired theory that arises explicitly from the low-energy effective action of superstring theory \cite{Clifton2012,Gross1987,NOVELLO2008}, and it has demonstrated success in several areas, including explaining cosmic acceleration \cite{Nojiri2005,Nojiri2006} and modeling cosmic inflation \cite{Kanti2015,Kanti2015a,Hikmawan2016}.

	A major goal of modified gravity theories is to construct cosmological models that evade singularities, which are unavoidable in classical general relativity and mark the theory’s breakdown (see, e.g.,  \cite{NOVELLO2008,Kawai1998,KalyanaRama1997,Wang2019a}). In the Einstein–scalar–Gauss–Bonnet (EsGB) framework, which describes the universe’s past and future evolution, only a few exact analytical solutions are currently known \cite{he2025proofssingularityfreesolutionsscalarization,Kanti2015,Kanti2015a,Hikmawan2016,NOVELLO2008}. The companion paper \cite{he2025proofssingularityfreesolutionsscalarization} and the present work together establish families of singularity-free solutions within the EsGB and EdGB frameworks. Specifically, the companion paper addresses the \textit{quadratic coupling} system, while this paper focuses on the \textit{exponential coupling} system, corresponding to dilaton-like scalar fields. \textit{We point out that the exponential coupling introduces several new analytical challenges that are absent in the quadratic coupling system.}

	This work is significant in three main aspects. 
	\begin{enumerate}[leftmargin=*]
		\item It provides a rigorous proof of the existence of singularity-free solutions in EdGB cosmology with exponential coupling, with results consistent with numerical simulations;
		\item It introduces a novel analytical technique, referred to as the \textit{power identity}, which is developed and applied in both the companion paper \cite{he2025proofssingularityfreesolutionsscalarization} and this work;
		\item The use of an exponential coupling function is motivated by the form of the dilaton coupling that appears in string-inspired effective actions. In this context, the scalar field in the EdGB model can be viewed as playing the role of a dilaton, which provides a well-grounded theoretical basis for constructing and analyzing singularity-free solutions in EdGB cosmology. 
	\end{enumerate}
	
	To the best of our knowledge, prior studies on this topic have been primarily numerical, and no rigorous analytical proof has been provided. The companion paper and the present one constitute the first rigorous mathematical demonstrations of singularity-free solutions in string-inspired cosmology, with results in full agreement with numerical findings.

	\subsection{Main Theorem}
	
	We consider the action (see, e.g., \cite{he2025proofssingularityfreesolutionsscalarization,Rizos1994,Kanti1999,Kanti2015,Kanti2015a,Kawai1999,Sberna2017a,Hikmawan2016})
	\begin{equation}\label{e:S}
		S_{\text{EsGB}} = \frac{1}{16\pi} \int d^4x \sqrt{-g} \left(\frac{1}{2} R - \frac{1}{2}\partial_\mu\phi\partial^\mu\phi - V_\phi - \lambda \frac{f(\phi)}{8} R^2_{\text{GB}} \right),
	\end{equation}
	where $R$ is the scalar curvature, $V_\phi$ the scalar potential,   $\lambda$ the coupling constant, $\phi$ the dilaton-like field, and $f(\phi)$ the coupling function.
	
	In this work, we focus on a free scalar field with positive coupling constant and an \textit{exponential coupling function} (i.e., EdGB gravity), normalized as
	\begin{equation*}
		V_\phi=0,\quad \lambda=1 \quad \text{and} \quad f(\phi) = e^\phi.
	\end{equation*}
	
	We consider a homogeneous and isotropic universe described by the \textit{Friedmann-Lemaître-Robertson-Walker (FLRW)} metric,
	\begin{equation}\label{e:FLRW}
		g(t) = -dt^2 + a^2(t) \left(  dr^2 + r^2(d\theta^2 + \sin^2\theta \, d\varphi^2) \right),
	\end{equation}
	where  
	\begin{equation}\label{e:aH}
		a(t):=a_0 e^{\int^t_0H(s)ds}
	\end{equation}
	is the scale factor, $H=\dot{a}/a$ denotes the Hubble parameter, and the dot represents a derivative with respect to time. Substituting the metric \eqref{e:FLRW} into the action \eqref{e:S} yields the system (see \S\ref{sec:2} for details).
	\begin{align}
		&3H^2-3e^\phi\dot \phi H^3=\frac{\dot \phi^2}{2},
		\label{eq:2.5} \\ 
		& 2\dot{H}+3H^2=-\frac{\dot{\phi}^2}{2}+2e^\phi \dot{\phi}H(H^2+\dot H)+e^\phi H^2(\dot \phi ^2+\phi \ddot{\phi}),
		\label{eq:2.4} \\
		&\ddot{\phi}=-3H\dot{\phi}-3e^\phi H^2(H^2+\dot{H}).
		\label{eq:2.6}
	\end{align} 
	Our analysis focuses on this nonlinear ODE system.

	\begin{theorem} \label{theorem:2.2}
		Suppose the initial data
		\begin{equation}\label{initial-data!!}
			(a_0,\beta,\alpha) := (a,H,\phi)|_{t=0} 
		\end{equation}
		satisfy 
		\begin{gather}
			a_0\in(0,+\infty) , \quad \alpha=0,\quad  \beta  \in \biggl(0,\frac{\sqrt{6}}{6}\biggr),  \label{eq:2.12!!}
		\end{gather} 
		and the condition\footnote{Note that \eqref{eq:2.5}, known as the \textit{Hamiltonian constraint}, is quadratic in 
			$\dot \phi$, giving two algebraic solution branches of $\dot\phi$. Only the negative branch \eqref{eq:2.11!!} leads to singularity-free solutions.}
		\begin{equation}
			\dot{\phi}(0) < 0, \label{eq:2.11!!}
		\end{equation} 
		then there exists a unique \textit{globally singularity-free, homogeneous, and isotropic FLRW solution} $(g,\phi)\in C^2((-\infty, +\infty))$, where $g$ is defined by \eqref{e:FLRW}, solving the EdGB field equations \eqref{eq:2.5}--\eqref{eq:2.6} with initial data \eqref{initial-data!!}.  
		Moreover, the Hubble parameter $H$ and dilaton-like field $\phi$ of this FLRW solution satisfy the following estimates (a schematic diagram is shown in Fig. \ref{fig:bound})
		\begin{enumerate}[leftmargin=*]
			\item For $t \in (-\infty,0)$, 	$H$ satisfies
		\begin{align*}
			H(t)< & \min\left\{ \left(e^{6 \gamma  t} \left(\frac{1}{\beta ^2}-\frac{4 \beta }{5 \gamma }+\frac{2 \theta }{15 \gamma ^2}\right)-\frac{2 (-6 \beta  \gamma +\theta +6 \gamma  \theta  t)}{15 \gamma ^2}\right)^{-\frac{1}{2} },  \;\frac{\gamma}{-3\beta^3 t+1} \right\}   , \\
			H(t)> &   \left[ \beta^{-1/4} - \frac{\beta^{-1/4}}{4(m+1)} \left( 1 - (-3\beta^3 t+1)^{m+1} \right) \right]^{-4}>0    , 
		\end{align*}
		and $\phi$ satisfies
		\begin{equation*} 
			\ln (-3\beta^3t+1)^2 < \phi < \frac{\sqrt{6} \gamma}{3\beta^3} \ln(-3\beta^3 t+1) +   \frac{2\beta^3}{4\gamma + \beta^3} \left[ (-3\beta^3 t+1)^{\frac{4\gamma + \beta^3}{\beta^3}} - 1 \right]  , 
		\end{equation*} 
			where 		\begin{equation*}
			\theta:= \frac{2\big(9\beta^{6}+\sqrt{9\beta^{6}+12}\,\beta^{3}\big)}{\big(\sqrt{9\beta^{6}+12}+3\beta^{3}\big)^{2}} , \quad 			\gamma =  \frac{3\beta^3+\sqrt{9\beta^6+12}}{2} \AND 	m =\frac{107 \sqrt{9 \beta ^6+12}+309 \beta ^3}{60 \beta ^3}  . 
		\end{equation*} 
			\item For $t \in (0,+\infty)$, $H$ satisfies
			\begin{gather*}
				\frac{1}{5t+\frac{1}{\beta}} < H(t) < \frac{1}{\frac{1}{2}t+\frac{1}{\beta}} 
				\intertext{and $\phi$ satisfies}
				-(12\beta^2+2\sqrt{6})\ln{\biggl(\frac{1}{2}\beta t+1\biggr)}<\phi(t)<-\ln{\left(1+\frac{3}{5}\left(\beta^2-\frac{1}{(5t+\frac{1}{\beta})^2}\right)\right)}.
			\end{gather*}
		\end{enumerate} 
	\end{theorem}

	\begin{figure}[H]
		\centering
		\begin{subfigure}[b]{0.48\textwidth}
			\includegraphics[width=\linewidth]{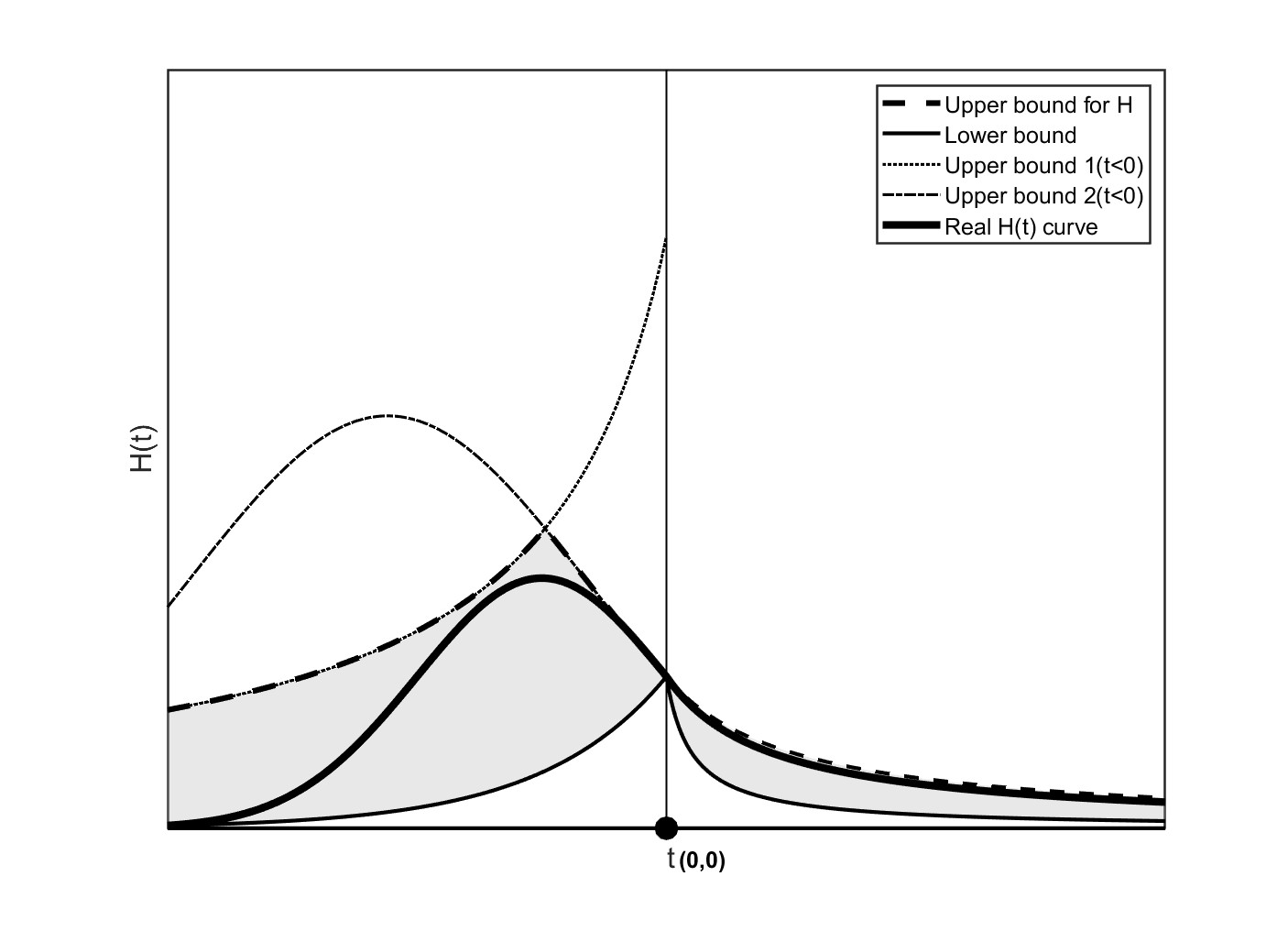}
			\caption{Bounds for $H$ } \label{fig:bound1}
		\end{subfigure}
		\begin{subfigure}[b]{0.48\textwidth}
			\includegraphics[width=\linewidth]{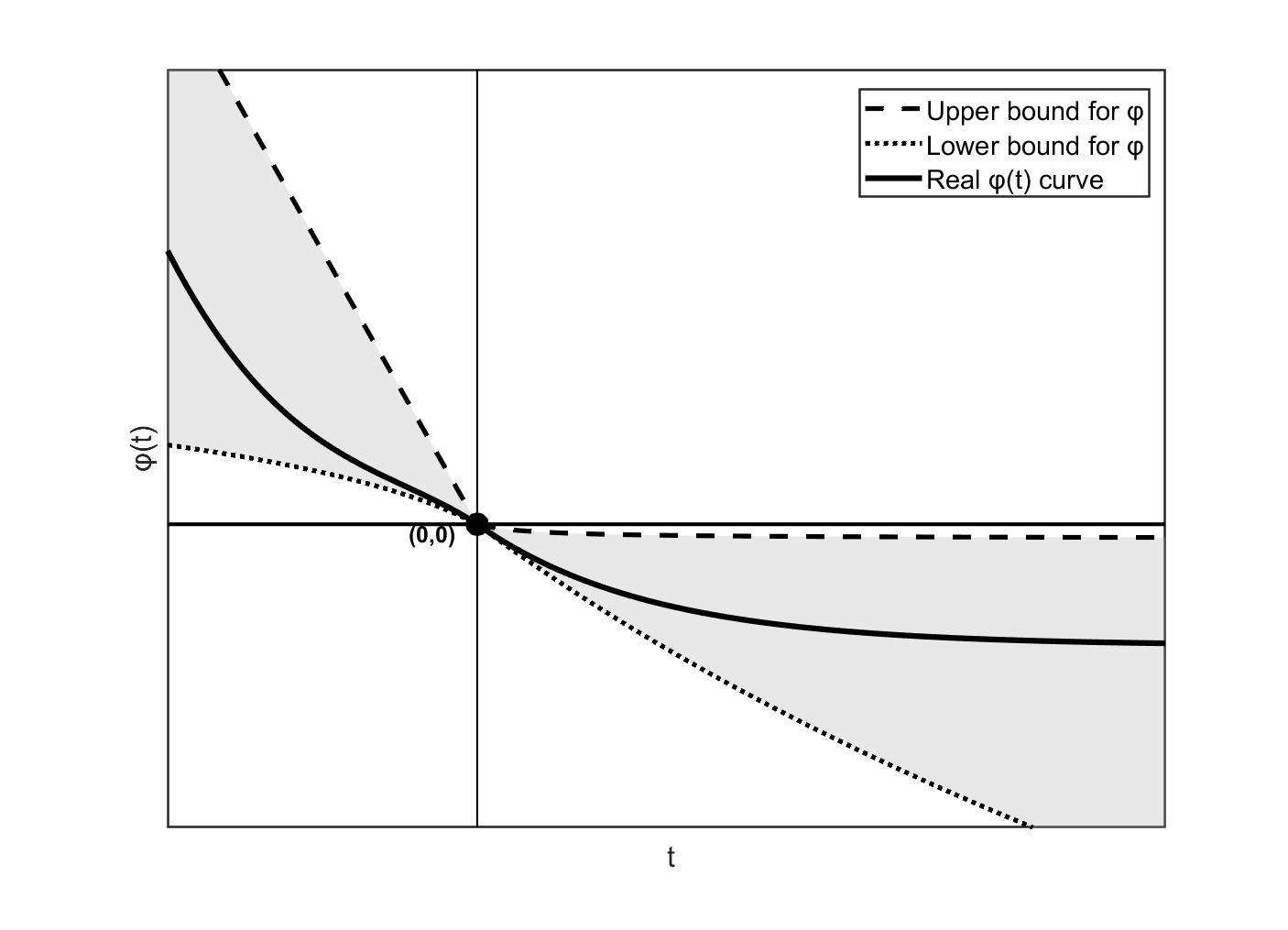}
			\caption{Bounds for $\phi$ }   \label{fig:bound2}
		\end{subfigure}
		\caption{Bounds for $H$ and $\phi$} \label{fig:bound}
	\end{figure}

	\begin{remark}
		In the schematic diagram shown in Fig.~\ref{fig:bound}, the shaded region corresponds to the estimates on $H$ and $\phi$ established in Theorem~\ref{theorem:2.2}, while the solid curves display the numerical solutions obtained with initial conditions $\alpha = 0$, $\beta = \tfrac{1}{5}$, and $\dot{\phi}(0) < 0$. The numerical results suggest that
		\begin{enumerate}
			\item solutions exist for all time $t \in (-\infty, +\infty)$;
			\item the Hubble parameter $H$ remains positive and asymptotically approaches zero as $t \to \pm\infty$;
			\item the dilaton-like field $\phi$ decreases monotonically and crosses zero at $t = 0$.
		\end{enumerate}
		The similar numerical findings can also be found in \cite[p.~81, Fig.~5.3]{Sberna2017a}.
	\end{remark}

	\subsection{Related Works}
	This work runs parallel to our companion article on singularity-free solutions and spontaneous scalarization in quadratic coupling theories \cite{he2025proofssingularityfreesolutionsscalarization}, but the case of exponential coupling presents additional technical challenges. While both studies use the power identity method to bound the Hubble parameter, the companion article focuses on scalarization in the quadratic regime, where the method allows \textit{decoupling the differential inequality for 
		$H$}. In contrast, the exponential coupling introduces more complicated nonlinearities, \textit{preventing such decoupling} and leading to more intricate nonlinear interactions.

	The exponential coupling in FLRW spacetime has been investigated in both previous numerical and linearized analytical studies.  For instance, numerical simulations for EdGB gravity in cosmological settings were presented in \cite{Sberna2017a}, and analytical investigations of string-inspired actions with coupling functions of the form 
	$e^{|\phi|}$
	(in contrast to our $e^\phi$ model) have established asymptotic solutions via linearization techniques \cite{Alexeyev2000, Kanti1999, Easther1996}.

	In the current work, we present a rigorous analytical proof addressing the fully nonlinear regime under exponential coupling.

	In addition, we point out that Eq. \eqref{eq:2.6} reveals a connection to the phenomenon of tachyonic instability. The linear tachyonic instability is also known as the Jeans instability in certain physical contexts. At the linear level, both instabilities arise from the same underlying mechanism, namely an effective negative mass squared in the perturbation equations. However, their nonlinear evolutions are dictated by the respective physical models, leading to distinct behaviors due to the differing nonlinear terms. A series of works on the nonlinear Jeans instability have been carried out by one of the authors (see \cite{Liu2022, Liu2022b, Liu2023a, Liu2023b, Liu2024}).

	\subsection{Outlines and Overview of Methods} 
		
	The main goal of this paper is to prove the existence of cosmological solutions in EdGB gravity that are free from the ``Big Bang'' singularity. This means the universe can be extended infinitely backward in time (as $t \rightarrow -\infty $) without the curvature becoming infinite.

	\subsubsection{Core Strategies and Key Techniques}\label{s:trategy}
	The proof starts with the equations describing cosmic evolution \eqref{eq:2.5}--\eqref{eq:2.6}. A key initial finding is that the rate of change of the scalar field, $\dot{\phi} $, can be directly expressed in terms of the Hubble parameter $H $ and the dilaton-like field $\phi$ itself (see \eqref{e:dotphi}). 
	This results in two possible initial branches, but the paper focuses on the ``negative branch'' where $ \dot{\phi} < 0 $, that is,\begin{equation}\label{e:dotphi0}
		\dot{\phi} = -3 H^3 e^{\phi} - \sqrt{(3 H^3 e^{\phi})^2 + 6 H^2}, 
	\end{equation}
	as it is promising for avoiding a singularity.
    
	\underline{Key tool: The power identity.}
	We derive a crucial structural identity (called the \textit{power identity})
	\begin{equation}\label{e:PI0} 
		\mathcal{P} := H\left( \left(1 - H^2 e^{\phi} + \frac{\dot{H}}{3 H^2}\right) \dot{\phi}^2 + 4 H^3 \dot{\phi} e^{\phi} + 3 H^6 e^{2\phi} \left(1 + \frac{\dot{H}}{H^2}\right) \right) = 0.
	\end{equation}
	This identity is not viewed as a differential equation but an algebraic relationship that must be satisfied by the solutions at all times. Its primary role is to act as a ``\textit{litmus test}'': at critical steps in the proof, by assuming that a desired property (e.g., $D_1:=\dot{H} + 5H^2>0$ on the interval of the existence) fails and substituting this assumption into the identity, one arrives at a contradiction like $0 < \mathcal{P} = 0$. This proves the initial assumption was false, meaning the good property must hold globally.
	
	\underline{The first hit argument.}
	This is the central technique used throughout the proof to demonstrate that certain quantities (like $ H(t) $) remain bounded for all time. The logic is straightforward, 
	\begin{enumerate}[label=\textbf{Step \arabic*:}, leftmargin=*, align=left]
		\item[\textit{Goal:}] Prove a quantity $ D_\ell(t) $ ($\ell=1,\cdots,5$) (e.g., $D_1:=\dot{H} + 5H^2$) is always \textit{positive} (or always negative).
		\item[\textit{Starting Point:}] First, show that $ D_\ell(0) > 0$ at the initial time (by $\dot{H}(0)$ or using the power identity).
		\item[\textit{Proof by Contradiction:}] Assume $ D_\ell(t)$ does not stay positive forever. Because these quantities change continuously, there must be a \textit{``first time''} $ T_{\text{max}} $ when $ D_\ell(t) $ hits zero.
		\item[\textit{Derive a Contradiction:}]  Evaluating the \textit{power identity} at $ t = T_{\text{max}} $ (i.e., $\mathcal{P}=0$) and plugging $D_\ell(T_{\text{max}} )=0$ into the \textit{power identity} (e.g., $D_1(T_{\text{max}} )=0$ implies $\dot H(T_{\text{max}})=-5H^2$) lead to an impossible conclusion (e.g., $ 0 < \mathcal{P} = 0 $).
		\item[\textit{Conclusion:}] Therefore, the ``first hit'' time $T_{\text{max}} $ cannot exist. So, $ D_\ell(t) $ must remain positive for all time. Then, for example, $D_1(t)>0$ implies a differential inequality that $\dot H>-5H^2$. Solving these differential inequalities derived from $D_\ell>0$, we are possible to conclude the estimates of $H$ in some situations (e.g. in the companion article \cite{he2025proofssingularityfreesolutionsscalarization} for $f(\phi)\sim \phi^2$, the quadratic coupling). However, there might be some new challenges in the current exponential coupling $f(\phi)=e^\phi$. 
		\item[\textit{New Challenges:}] 
		In the companion article \cite{he2025proofssingularityfreesolutionsscalarization}, the quadratic coupling allows us, via the power identity and the first-hit argument, to derive decoupled differential inequalities for $H$. This enables us to solve for $H$ first and subsequently for $\phi$. In contrast, the exponential coupling studied here presents a greater challenge. While a similar decoupling for $H$ remains possible in the future evolution ($t>0$), it fails for the past evolution ($t<0$), where the quantities $D_\ell$
		inherently depend on both $H$ and $\phi$. Consequently, a novel idea is required to analyze the backward-in-time behavior.
		 
	\end{enumerate}

	\subsubsection{The Two Regions of the Proof} 
\S\ref{sec:3} presents the main estimates for the solution. Since the behaviors in the future and past regimes differ substantially, we analyze them separately. Our aim is to prove the solution remains non-singular in the past and future direction by establishing upper and lower bounds for $H(t)$ and $\phi(t)$ for all $t\in(\mathcal{T}_-,\mathcal{T}_+)$. 

	\underline{(I) Analysis for $ t \in (0, \mathcal{T}_+) $ (Future Evolution)}
	The proof for the future is relatively direct, which is similar to the analysis in the companion \cite{he2025proofssingularityfreesolutionsscalarization}.
	\begin{enumerate}[leftmargin=*]
		\item Define two key quantities depending only on $H$
		\begin{equation*}
			D_1  = \dot{H} + 5H^2 \AND
			D_2  = \dot{H} + \frac{1}{2}H^2.
		\end{equation*}
		The motivation of selecting these definitions is they are engineered to interact productively with the power identity \eqref{e:PI0}. Their signs collectively control the growth of the solution.
		\item Use the power identity and the ``first hit argument'' to prove $ D_1(t) > 0 $ and $D_2(t) < 0 $ for all $ t \in (0, \mathcal{T}_{+}) $.
		\item By solving $ D_1(t) > 0 $ and $D_2(t) < 0 $ (Riccati-type), this directly implies that $ H(t) $ is ``sandwiched'' between two simple functions:
		\begin{equation*}
			\frac{1}{5t + \frac{1}{\beta}} < H(t) < \frac{1}{\frac{t}{2} + \frac{1}{\beta}}.
		\end{equation*}
		As $ t \rightarrow +\infty $, $ H(t) $ approaches zero smoothly, with no singularity.
		\item Once $H$ is determined, we can estimate $\phi$ using \eqref{e:dotphi0}.
	\end{enumerate}
	
	\underline{(II) Analysis for $ t \in (\mathcal{T}_-, 0) $ (Past Evolution)}
		The analysis for negative times is significantly more complex and constitutes the main technical challenge. The primary difficulty arises because $ \phi(t) \to +\infty $ as $ t \to -\infty $, making the exponential term $ e^{\phi} $ unbounded. The goal is to prove that despite this, $ H(t) $ remains positive and also vanishes asymptotically ($H \to 0 $ as $t \to -\infty $).
		
	\begin{enumerate}[leftmargin=*]
		\item \textit{Auxiliary Quantities:} Three more intricate quantities are defined, each intrinsically \textit{coupled} and involving both $H$ and $\phi$
		\begin{align*}
			D_3 &= \dot{H} - 3H^4 e^{\phi} + H^2, 
			\\
			D_4 &= \dot{H} - \frac{12}{5}H^4 e^{\phi} + 3H^2, 
			\\
			D_5 &= \dot{H} - H^2 + 3e^{-\phi}.  
		\end{align*}
In contrast, $D_1$ and $D_2$ depend \textit{only} on $H$. This \textit{additional coupling} in $D_3$, $D_4$, and $D_5$ makes the corresponding differential inequalities significantly more challenging to analyze.

		\item \textit{Sign Preservation of Auxiliary Quantities:} The power identity and the ``first hit  argument'' is again used to prove that these retain their specific signs ($ D_3 < 0 $, $ D_4 > 0 $, $ D_5 > 0 $) for all past time $ t \in (\mathcal{T}_{-}, 0) $.  This step is crucial as these inequalities provide the differential constraints necessary for the subsequent analysis (see Lemmas \ref{2}--\ref{3} for details).

		 \item \textit{Derivation of Differential Inequalities.} 
		 Using $ D_3 < 0 $, $ D_4 > 0 $, and $ D_5 > 0 $, we first obtain differential inequalities for $\dot{H} $ (e.g., 	$D_3  = \dot{H} - 3H^4 e^{\phi} + H^2<0$). However, because of the coupling and the \textit{unbounded} nature of $e^{\phi} $, these inequalities are difficult to solve directly. To address this issue, we examine the behavior of the composite quantities $H^3 e^{\phi} $, $ H e^{\phi} $, and $ H^{11/4} e^{\phi} $. Instead of attempting to control $ H $ through the unbounded exponential $ e^{\phi} $, our approach is to control $ H $ using the \textit{better-behaved} quantities $ H e^{\phi} $ and $ H^{11/4} e^{\phi} $, while controlling $ e^{\phi} $ through $ H^3 e^{\phi} $. Consequently, we introduce differential inequalities for $ H e^{\phi} $, $ H^{11/4} e^{\phi} $, and $ H^3 e^{\phi} $. These inequalities follow from $ D_3 < 0 $, $ D_4 > 0 $, and $ D_5 > 0 $, together with the bounds for $ \dot{\phi} $ provided in Corollary~\ref{dot phi}, namely
		$ -6 H^3 e^{\phi} - \sqrt{6}\,H < \dot{\phi} < -6 H^3 e^{\phi}$. 
		 This yields a coupled system of differential inequalities linking $ H $,  $\phi$, $ H e^{\phi} $, $ H^{11/4} e^{\phi} $, and $ H^3 e^{\phi} $ (see Step $1$ in the proof of Lemma \ref{t:Hphibd} for details).

		 \item \textit{Variable Transformation and Hierarchical Estimates.}
		 To handle the strong nonlinearity, a \textit{hierarchy of estimates} is constructed.
New variables are introduced to simplify the coupled system of differential inequalities, 
		 	\begin{align*}
		 		y  = H, \quad w = e^{\phi},  \quad z = H^3 e^{\phi},  \quad 
		 		v = H e^{\phi} \AND  p  = H^{11/4} e^{\phi}.
		 	\end{align*}
The system is presented in Step $2$ of the proof of Lemma \ref{t:Hphibd}. The inequalities are solved in a specific, \textit{hierarchical} order, using comparison theorems to obtain explicit bounds. Bounds for one variable are then used to establish bounds for the next, forming a sequential chain of estimates (see Fig. \ref{f:flowchart}). 
		 	\begin{enumerate}
		 		\item A lower bound for $z$ provides a lower bound for $w$ (and hence for $\phi$).
		 		\item The lower bound for $w$ helps establish an upper bound for $y$ (and thus $H$).
		 		\item This upper bound for $y$, combined with the lower bound for $z$, yields a lower bound for $v$.
		 		\item The lower bound for $v$, together with the upper bound for $y$, leads to an additional upper bound for $y$. Merging this with the first upper bound for $y$ gives an improved upper bound for $y$.  
		 		\item The upper bound for $y$ allows us to derive an upper bound for $z$. 
		 		\item The upper bound for $z$ and $y$ then lead to an upper bound for $w$ (equivalently, for $\phi$). 
		 		\item The lower bound for $z$ combined with the upper bound for $y$ yields an upper bound for $p$. 
		 		\item This upper bound for $p$ finally provides a lower bound for $y$ (that is, for $H$). 
		 		\item The process can be  continued, successively refining bounds for $y$ $z$, $w$ ($\phi$), $p$ and $v$. However, we do not pursue further refinements.
		 	\end{enumerate}
This step-by-step hierarchical strategy is crucial for handling the nonlinear coupling.

		 \item \textit{Final Bounds.}
		 	The hierarchical analysis culminates in Proposition \ref{proposition:8}, which provides explicit, time-dependent bounds valid for all $t \in (\mathcal{T}_-, 0)$.  
	\end{enumerate}

    	\begin{figure}[H]
	\centering 
		\includegraphics[width=0.9\textwidth]{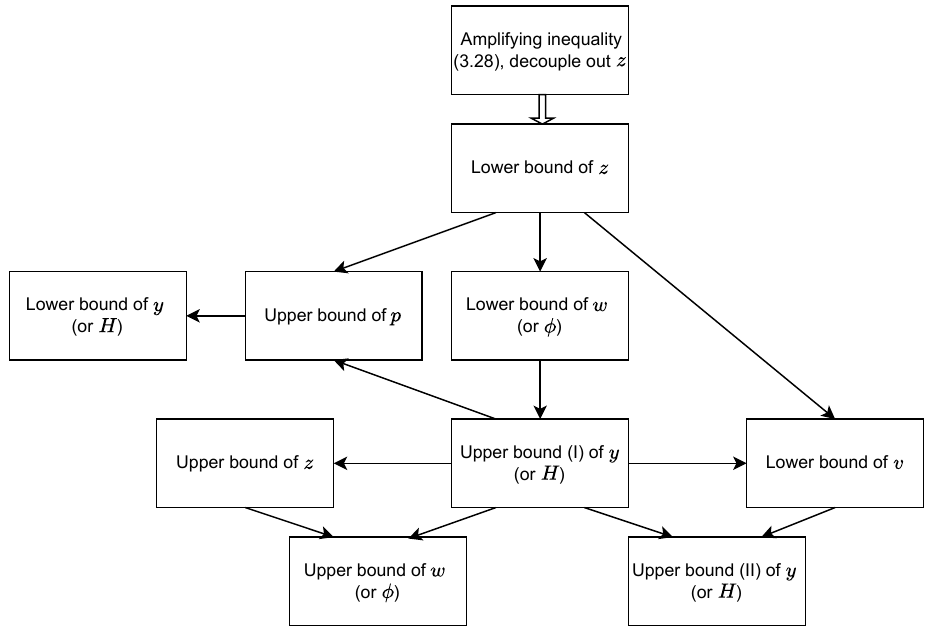}
		\caption{Hierarchical Estimates} \label{f:flowchart}  
\end{figure}
	
	\subsubsection{Global Existence}
	After obtaining upper and lower bounds for $ H(t) $ and $ \phi(t) $ on finite time intervals, \S\ref{sec:4} finally argues that the solution can be extended indefinitely into the infinite past and future, i.e., $ \mathcal{T}_{-} = -\infty $ and $ \mathcal{T}_{+} = +\infty $.
	
	\subsubsection{Additional Remarks}
	Compared to the quadratic coupling 
 $f(\phi)\sim \phi^2$
	discussed in \cite{he2025proofssingularityfreesolutionsscalarization}, whose symmetry 
$\phi \rightarrow -\phi$ lets us focus on a single branch of the constraint equation, the exponential coupling 
$e^\phi$
	generates two distinct branches. Numerical results indicate that singularity-free solutions appear only on the negative branch.

	\subsubsection{Summary}
	In summary, the argument relies on a key algebraic relation (the \textit{power identity}) together with a strategic \textit{first-hit contradiction argument}. This approach converts the task of controlling the solution into the task of showing that the signs of several auxiliary quantities $D_1, D_2, D_3, D_4, D_5 $ are preserved. By carefully deriving the differential inequalities dictated by these sign conditions, we ultimately show that the central cosmological quantity, the Hubble parameter $H(t)$, remains finite for all time, thereby establishing a singularity-free universe. In particular, we prove that as 
$ t \rightarrow \pm \infty $, the Hubble parameter  $H(t)$ stays finite and in fact approaches zero. The behavior of the scalar field $ \phi(t) $ is controlled as well. Thus the model avoids a Big-Bang-type singularity.

	\section{Constraints and local existence}
	\label{sec:2}
	We analyze the EdGB action \eqref{e:S} within the framework of a homogeneous and isotropic universe, described by the FLRW metric \eqref{e:FLRW}. Under this assumption, the scalar curvature and the Gauss–Bonnet term become 
	\begin{equation*}
		R = 6 \left( \frac{\ddot{a}}{a} + \left( \frac{\dot{a}}{a} \right)^2 \right) \AND R^2_{GB} = 24 \left( \frac{\dot{a}}{a} \right)^2 \frac{\ddot{a}}{a}. 
	\end{equation*} 
	The system then can be described by the following equations for the Hubble parameter
	\begin{equation*}
		H := \frac{\dot{a}}{a} ,
	\end{equation*}
	and the dilaton-like field $\phi$
	\begin{equation}
		3H^2-3e^\phi\dot \phi H^3=\frac{\dot \phi^2}{2}, \quad 
		\label{eq:2.5++}
	\end{equation}
	\begin{equation}
		2\dot{H}+3H^2=-\frac{\dot{\phi}^2}{2}+2e^\phi \dot{\phi}H(H^2+\dot H)+e^\phi H^2(\dot \phi ^2+\phi \ddot{\phi}),
		\label{eq:2.4+}
	\end{equation}
	\begin{equation}
		\ddot{\phi}=-3H\dot{\phi}-3e^\phi H^2(H^2+\dot{H})  . 
		\label{eq:2.6++}
	\end{equation}
	We point out that \eqref{eq:2.5++} is known as the Friedmann equation, which is also  the Hamiltonian constraint equation.

	\subsection{Initial data and the constraint}\label{s:data} 
	In analogy with the Cauchy problem for Einstein–scalar systems (cf.~\cite{ChoquetBruhat2009,Ringstroem2009}),  \eqref{eq:2.5++} serves as the \emph{Hamiltonian constraint}, which restricts the admissible initial data $(a,H,\phi,\dot{\phi})|_{t=0}$. Once satisfied at $t=0$, this constraint remains preserved for all $t$, as equations~\eqref{eq:2.4+}–\eqref{eq:2.6++} imply that
	\begin{equation}\label{e:dtctnt}
		\partial_t \left(3H^2-3e^\phi\dot \phi H^3-\frac{\dot \phi^2}{2}\right) = 0.
	\end{equation}
	Solving \eqref{eq:2.5++} for $\dot{\phi}$ yields  
	\begin{equation}\label{e:dotphi}
		\dot{\phi}  =   -3 H^3e^\phi +(-1)^\iota \sqrt{\left(3 H^3e^\phi\right)^2 + 6H^2}, \quad  (\iota=0,1 )  . 
	\end{equation}

	Therefore, specifying the initial data $(\beta,\alpha) := (H, \phi)|_{t=0}$ determines $\dot{\phi}|_{t=0}$ via \eqref{e:dotphi}, that is, 
	\begin{equation*}
		\dot{\phi} |_{t=0} 
		=  -3e^\alpha \beta^3 +(-1)^\iota \sqrt{\left(3e^\alpha \beta^3\right)^2 + 6 \beta^2}, \quad  (\iota=0,1 ).
	\end{equation*}
	This gives two admissible initial data sets:
	\begin{align}\label{e:dtset}
		(H,\phi,\dot{\phi})|_{t=0} =\begin{cases}
			(\beta,\alpha,-3e^\alpha \beta^3 + \sqrt{\left(3e^\alpha \beta^3\right)^2 + 6 \beta^2}), \\
			(\beta,\alpha,-3e^\alpha \beta^3 - \sqrt{\left(3e^\alpha \beta^3\right)^2 + 6 \beta^2}).
		\end{cases}
	\end{align}

	\subsection{Local existence}
	
	We begin with a useful estimate:
	\begin{lemma}\label{t:Dvalue}
		If $\dot{\phi}$ satisfies \eqref{e:dotphi} on some set $\mathcal{I}$, then for all $t \in \mathcal{I}$,
		\begin{equation*}
			2-2He^\phi\dot{\phi}+3H^4e^{2\phi} >0. 
		\end{equation*}  
	\end{lemma}
	
	\begin{proof}
		(1) If $H = 0$ at some $t \in \mathcal{I}$, then the expression reduces to $2-2He^\phi\dot{\phi}+3H^4e^{2\phi}=2 > 0$ at those points.
		
		(2) If $H \neq 0$ at $t \in \mathcal{I}$, then using \eqref{e:dotphi}, we have
		\begin{equation*}
			He^\phi\dot{\phi} = -3 H^4e^{2\phi} \pm \sqrt{\left(3 H^4e^{2\phi}\right)^2 + 6H^4e^{2\phi}}.
		\end{equation*} 
		Applying \eqref{eq:B2}, it follows that
		\begin{align*}
			He^\phi\dot{\phi}  = &  -3 H^4e^{2\phi} \pm \sqrt{\left(3 H^4e^{2\phi}\right)^2 + 6H^4e^{2\phi}} \notag \\
			\leq  &  -3 H^4e^{2\phi} + \sqrt{\left(3 H^4e^{2\phi}\right)^2 + 6H^4e^{2\phi}}<\sqrt{6}H^2e^\phi . 
		\end{align*}
		Substituting this into the expression yields
		\begin{equation*}
			2-2He^\phi\dot{\phi}+3H^4e^{2\phi}>	2-2\sqrt{6}H^2e^\phi+3H^4e^{2\phi} = (\sqrt{2}-\sqrt{3}H^2e^\phi)^2 \geq 0 . 
		\end{equation*}
		
		This completes the proof.
	\end{proof}

	\begin{proposition}[Local Existence]\label{t:locext} 
		Given initial data $(a_0, \beta, \alpha) := (a,H,\phi)|_{t=0}$, there exists a constant $T > 0$, such that the system \eqref{eq:2.4+}--\eqref{eq:2.6++} with this initial data admits a pair of solutions $(a,H,\phi,\dot{\phi}) \in C^1((-T,T), \mathbb{R}^4)$ corresponding to the initial data sets in \eqref{e:dtset}. 
	\end{proposition}
	\begin{proof} 
		Substitute \eqref{eq:2.6++} into \eqref{eq:2.4+}. Using Lemma \ref{t:Dvalue}, we obtain
		\begin{equation}\label{eq-dotH}
			\dot{H} = \frac{-4H^3 \dot{\phi}e^\phi - \dot{\phi}^2 + H^2e^\phi \dot{\phi}^2 - 3H^6 e^{2\phi}}{2-2He^\phi\dot{\phi}+3H^4e^{2\phi}}.
		\end{equation}  
		Let $\Phi:=\dot{\phi}$. 
		Then gathering  \eqref{eq-dotH}, $\dot{\phi} = \Phi$ and \eqref{eq:2.6++} together, the system becomes
		\begin{equation}\label{e:locsys} 
			\frac{d}{dt} \p{H\\\phi\\\Phi}= \p{ F_1(H,\phi,\Phi) \\
				\Phi \\ F_2(H,\phi,\Phi) },
		\end{equation}
		where
		\begin{align}
			F_1(H,\phi,\Phi):=& \frac{-4H^3 \Phi e^\phi - \Phi^2 + H^2e^\phi \Phi^2 - 3H^6 e^{2\phi}}{2-2He^\phi\Phi+3H^4e^{2\phi}},  \label{e:F1} \\
			F_2(H,\phi,\Phi):=& - 3\Phi H - 3H^2 e^\phi \left(H^2 + F_1(H,\phi,\Phi) \right). \label{e:F2} 
		\end{align}
		We can verify that $F_1, F_2\in C^1(\Rbb^3)$. By \S\ref{s:data}, the initial data set \eqref{e:dtset} provides two admissible choices. Applying Theorem~\ref{theorem:PL} and \eqref{e:aH} ($H=\dot{a}/a$), we obtain a unique solution $(a,H, \phi, \Phi) \in C^1((-T, T),  \mathbb{R}^4)$ to the system \eqref{e:locsys} for each choice of initial data. Hence, the original system \eqref{eq:2.4+}–\eqref{eq:2.6++} admits a pair of $C^1$ solutions $(a,H, \phi, \dot{\phi})$ on $(-T, T)$ corresponding to the initial data $(a_0,\beta,\alpha)$. This completes the proof. 
	\end{proof}

	\section{Estimates of the FLRW Solution in the EsGB System}
	\label{sec:3} 
	In the previous section, the local existence of a FLRW solution $( H, \phi, \dot{\phi}) \in C^1((-T, T), \mathbb{R}^3)$ to the EdGB system \eqref{eq:2.5}--\eqref{eq:2.6} with initial data $( \beta, \alpha)$ has been proven. In this section, we consider the bounds for this solution within its interval of existence. 
	
	We first denote $\mathcal{T}_- \in [-\infty, 0)$ and $\mathcal{T}_+ \in (0, +\infty]$ as the maximal backward and forward existence times of the FLRW solution, respectively.

	Before proceeding, we point out
	\begin{remark} 
		We set the initial condition as $\alpha = 0$. In this case, from \eqref{e:dotphi}, we obtain
		\begin{equation}\label{e:dotphsgn}
			\dot \phi(0)=\begin{cases}
				-3\beta^3 + \sqrt{9\beta^6 + 6 \beta^2}>0, \\
				-3\beta^3 - \sqrt{9\beta^6 + 6 \beta^2}<0.
			\end{cases}
		\end{equation} 
		To establish Theorem~\ref{theorem:2.2}, we focus on the \textit{negative branch} (i.e., $\iota = 1$) in \eqref{e:dotphi}, since in this case, by \eqref{e:dotphsgn}, the initial data are consistent with the condition given in \eqref{eq:2.11!!}. 
	\end{remark}

	\begin{proposition}
		\label{lemma:3.1}
		If $H(t)>0$ and $\dot{\phi}(0) < 0$, the solution $\phi$ satisfies
		\begin{equation*} 
			\dot{\phi}(t) = -3  H^3(t)e^{\phi(t)} - \sqrt{\left(3 H^3(t)e^{\phi(t)}\right)^2 + 6 H^2(t)}<0
		\end{equation*}
		for all $t \in (\mathcal{T}_-,\mathcal{T}_+)$. 
	\end{proposition} 
	\begin{proof} 
		Note the fact that $H(t)>0$ implies $\dot{\phi}(t)\neq 0$ and further
		\begin{equation*}
			\sqrt{\left(3 H^3(t)e^{\phi(t)}\right)^2 + 6 H^2(t)}>3  H^3(t)e^{\phi(t)}.
		\end{equation*}
		
		From \eqref{e:dotphi} and the fact $e^\phi>0$, we obtain for $t \in (\mathcal{T}_-, \mathcal{T}_+)$, 
		\begin{equation}\label{e:dphisgn}
			\dot{\phi}(t)
			\begin{cases}
				>0, \quad \text{for } \iota=0, \\
				<0, \quad \text{for } \iota=1. 
			\end{cases}
		\end{equation} 
		Since $\dot{\phi} \in C^1((\mathcal{T}_-, \mathcal{T}_+))$, the local existence theorem (Theorem \ref{t:locext}) guarantees that if $\dot{\phi}(0) < 0$, there exists a constant $T > 0$ such that $\dot{\phi}(t) < 0$ for $t \in (-T, T)$. 
		Therefore, $\dot{\phi}$ must be given by the $\iota = 1$ branch of \eqref{e:dotphi} throughout $(\mathcal{T}_-, \mathcal{T}_+)$; otherwise, a sign change would contradict the continuity of $\dot{\phi}$ implied by \eqref{e:dphisgn}. This completes the proof. 
	\end{proof}

	As we focus on the negative branch, $\dot{\phi}$ can be further constrained as follows
	\begin{corollary}
		\label{dot phi}
		If $H(t)>0$ and $\dot{\phi}(0) < 0$, then $\dot\phi$ satisfies
		\begin{equation*}
			-6H^3e^\phi-\sqrt{6}H<\dot{\phi}<-6H^3e^\phi . 
		\end{equation*}
	\end{corollary}
	\begin{proof}
		The result can be obtained immediately from Proposition \ref{lemma:3.1} and the inequalities \eqref{eq:B2} and \eqref{eq:B3} in Lemma \ref{t:bscineq}.
	\end{proof}

	Next, we present the main tool of this article. 
	\begin{lemma}[Power identity]
		\label{proposition:3.2++}
		The following identity holds for all $t\in(\mathcal{T}_-,\mathcal{T}_+)$:
		\begin{equation*} \mathcal{P}:=H\left(\left(1-H^2e^\phi+\frac{\dot H}{3H^2}\right)\dot{\phi}^2+4H^3\dot \phi e^\phi+3H^6e^{2\phi}\left(1+\frac{\dot{H}}{H^2}\right) \right) = 0. 
		\end{equation*}
	\end{lemma}
	\begin{proof}
		We begin by expanding \eqref{e:dtctnt}, and then use \eqref{eq:2.6++} to eliminate the second derivative $\ddot{\phi}$. This yields
		\begin{equation*}
			H\left((1-H^2e^\phi)\dot{\phi}^2+4H^3\dot \phi e^\phi+3H^6e^{2\phi}\left(1+\frac{\dot{H}}{H^2}\right)+2\dot{H}(1-H\dot{\phi}e^\phi)\right) = 0. 
		\end{equation*}
		Finally, substituting the expression from \eqref{eq:2.5++} to replace $1-H\dot{\phi}e^\phi$ completes the derivation. 
	\end{proof}

	\subsection{Bounds of $H$ for $t \in (0,\mathcal{T}_+)$}
	We begin by establishing bounds for $H$ on the interval $t \in (0, \mathcal{T}_+)$. We define two useful quantities $D_1(t)$ and $D_2(t)$, 
	\begin{align}
		D_1:=&\dot{H}+5H^2; 
		\label{D_1}  \\
		D_2:=&\dot{H}+\frac{1}{2}H^2.
		\label{D_2}
	\end{align}

	In what follows, we aim to establish the lower bound of $H(t)$ for $t \in (0, \mathcal{T}_+)$ in Proposition \ref{proposition:5}.

	\begin{lemma}\label{lemma:3.2!}
		If there exists a constant $T_0 > 0$ such that $D_1(t) > 0$ for all $t \in (0, T_0)$, then
		\begin{equation*}
			H(t) > \frac{1}{5t + \frac{1}{\beta}}, \quad \text{for all } t \in (0,T_0).
		\end{equation*}
	\end{lemma}
	\begin{proof}
		The assumption implies that $H$ satisfies, for all $t \in (0,T_0)$, 
		\begin{align*}
			\dot{H}>-5H^2 \quad \text{and} \quad
			H|_{t=0}=\beta >0.
		\end{align*}  
		Integrating the above Riccati-type inequality, Theorem \ref{theorem:C} guarantees the following bound holds
		\begin{equation*}
			H(t) > \frac{1}{5t + \frac{1}{\beta}}, \quad \text{for all } t \in (0,T_0), 
		\end{equation*}
		which completes the proof. 
	\end{proof}

	\begin{lemma}\label{lemma:3.!}
		If $\beta>0$ and $D_1(0) > 0$, then $D_1(t) > 0$ for all $t \in (0, \mathcal{T}_+)$.
	\end{lemma}
	\begin{proof}
		Since $D_1(0) > 0$ and $D_1(t)$ is continuous (by the definition \eqref{D_1} and Proposition \ref{t:locext}), there exists a constant $T\in(0, \mathcal{T}_+]$, such that $D_1(t)>0$ for all $t\in(0, T)$. Define
		\begin{equation}\label{e:Tmax!!}
			T_\text{max}:=\sup\{T\in(0,\mathcal{T}_+] \;|\;D_1(t)>0 \;\;\text{for all } t \in (0, T)\}.
		\end{equation}
		
		To prove the lemma, it suffices to show that $T_{\text{max}} = \mathcal{T}_+$. Suppose, for contradiction, that $T_{\text{max}} < \mathcal{T}_+$. Then because of the continuity of $D_1$, we must have $D_1(T_{\text{max}}) = 0$; otherwise, this would contradict the definition of $T_{\text{max}}$ in \eqref{e:Tmax!!}.
		
		By Lemma \ref{lemma:3.2!}, the positivity of $D_1$ on $(0,T_{\text{max}})$ implies
		\begin{equation*}
			H(t) > \frac{1}{5t + \frac{1}{\beta}}, \quad \text{for all } t \in (0, T_{\text{max}}).
		\end{equation*} 
		Taking the limit as $t \to T^-_{\text{max}}$ and using continuity of $H$, we obtain
		\begin{equation*}
			H(T_{\text{max}}) \geq \frac{1}{5T_{\text{max}} + \frac{1}{\beta}} > 0,
		\end{equation*}
		where the last inequality follows from the fact that $T_{\text{max}} <\mathcal{T}_+\in(0,+\infty]$.
		
		Next, evaluate the power identity at $t = T_{\text{max}}$. By Lemma \ref{proposition:3.2++}, we have $\mathcal{P}(T_{\text{max}}) = 0$, and taking $D_1(T_{\text{max}}) = 0$ into account yields 
		\begin{align}
			0=&\mathcal{P}(T_{\text{max}} )=\left. H \Biggl( \biggl(-H^2e^\phi-\frac{2}{3}\biggr)\dot{\phi}^2+4H^3\dot{\phi}e^\phi-12H^6e^{2\phi} \Biggr) \right|_{t=T_{\text{max}} } \notag \\
			=&\left. H \Biggl( -H^2e^\phi\dot{\phi}^2-\frac{1}{3}\dot{\phi}^2-\biggl(\frac{\sqrt{3}}{3}\dot{\phi}-2\sqrt{3}H^3e^\phi\biggr)^2  \Biggr) \right|_{t=T_{\text{max}} } <0 . \label{e:0>0}
		\end{align}  
		Since all the terms in the bracket are negative and $H(T_{\text{max}}) > 0$, the entire expression is strictly negative, leading to the contradiction \eqref{e:0>0}. 
		Hence, our assumption that $T_{\text{max}} < \mathcal{T}_+$ must be false, and we conclude that $T_{\text{max}} = \mathcal{T}_+$. This completes the proof. 
	\end{proof}

	\begin{lemma} \label{t:dataB!}
		Under the initial conditions \eqref{eq:2.12!!}--\eqref{eq:2.11!!}, we have $D_1(0) > 0$.
	\end{lemma}
	
	\begin{proof}
		Substituting \eqref{eq-dotH} into the definition of $D_1$ given in \eqref{D_1} yields
		\begin{align*}
			D_1(0)=(\dot{H}+5H^2)|_{t=0}&=\frac{-4H^3 \dot{\phi} +\frac{2}{3} \dot{\phi}^2 + H^2 \dot{\phi}^2 +12 H^6 }{2-2H\dot\phi+3H^4}\Biggr|_{t=0}\\
			&=\frac{H^2\dot{\phi} ^2+\frac{1}{3}\dot{\phi}^2+(\frac{\sqrt{3}}{3}\dot{\phi}-2\sqrt{3}H^3)^2}{2-2H\dot\phi+3H^4}\Biggr|_{t=0}>0.
		\end{align*}
		This concludes this lemma. 
	\end{proof}

	\begin{proposition}
		\label{proposition:5}
		Under the initial conditions \eqref{eq:2.12!!}--\eqref{eq:2.11!!}, we have
		\begin{equation*}
			H(t) > \frac{1}{5t + \frac{1}{\beta}} >0, \quad \text{for all } t \in (0, \mathcal{T}_+).
		\end{equation*}
	\end{proposition}
	
	\begin{proof}
		Lemma \ref{t:dataB!} implies $D_1(0)>0$, the result then follows directly from Lemmas \ref{lemma:3.2!} and \ref{lemma:3.!}.
	\end{proof}
	
	Now we turn to the upper bound of $H(t)$ for $t \in (0, \mathcal{T}_+)$ as stated in Proposition~\ref{proposition:6}. Throughout the following arguments in the rest of this section, we keep in mind that
$H>0$ by Proposition~\ref{proposition:5}, and we will not state this explicitly each time.

	\begin{lemma}\label{lemma:3.2!!!}
		If there exists a constant $T_0 > 0$ such that $D_2(t) < 0$ for all $t \in (0, T_0)$, then
		\begin{equation*}
			H(t) < \frac{1}{\frac{1}{2}t + \frac{1}{\beta}}, \quad \text{for all } t \in (0,T_0).
		\end{equation*}
	\end{lemma}
	\begin{proof}
		Since $D_2(t) < 0$, it follows that $H$ satisfies a Riccati-type inequality, for all $t \in (0, T_0)$,
		\begin{align*}
			\dot{H}<-\frac{1}{2}H^2 \quad \text{and} \quad
			H|_{t=0}=\beta.
		\end{align*}  
		Integrating the above inequality, Theorem \ref{theorem:C} guarantees the following bound holds
		\begin{equation*}
			H(t) < \frac{1}{\frac{1}{2}t + \frac{1}{\beta}}, \quad \text{for all } t \in (0,T_0), 
		\end{equation*}
		which completes the proof.
	\end{proof}

	\begin{lemma}\label{lemma:3.3!!!}
		Under the initial conditions \eqref{eq:2.12!!}--\eqref{eq:2.11!!}, we have $D_2(t) < 0$ for all $t \in (0, \mathcal{T}_+)$.
	\end{lemma}
	\begin{proof}
    Firstly, $D_2(0)<0$ can be derived from \eqref{eq:2.12!!}--\eqref{eq:2.11!!} by direct computations as the proof of Lemma \ref{t:dataB!}. Since $D_2(0) < 0$ and $D_2(t)$ is continuous (by the definition \eqref{D_2} and Proposition \ref{t:locext}), there exists a constant $T\in(0, \mathcal{T}_+]$, such that $D_2(t)<0$ for all $t\in(0, T)$. Define
		\begin{equation}\label{e:Tmax}
			T_\text{max}:=\sup\{T\in(0,\mathcal{T}_+] \;|\;D_2(t)<0 \;\;\text{for all } t \in (0, T)\}.
		\end{equation}
		
		To prove the lemma, it suffices to show that $T_{\text{max}} = \mathcal{T}_+$. Suppose, for contradiction, that $T_{\text{max}} < \mathcal{T}_+$. Then because of the continuity of $D_2$, we must have $D_2(T_{\text{max}}) = 0$, otherwise, this would contradict the definition of $T_{\text{max}}$ in \eqref{e:Tmax}.

In the next, we prove two useful inequalities. 	 
According to Corollary \ref{dot phi}, we have
\begin{equation}\label{e:dpng}
	\dot{\phi}<-6H^3e^\phi<0, 
\end{equation}
and consequently,
			\begin{equation}\label{e:dp^2}
				\dot{\phi}^2>-6H^3e^\phi\dot{\phi}   . 
			\end{equation}
Note, by the definition of $D_2$ given in \eqref{D_2}, together with $D_2\leq 0$ on $(0,T_\text{max}]$ and $\dot\phi<0$ given by \eqref{e:dpng}, in $(0,T_{\text{max}}]$, we obtain
			\begin{equation*}
				\frac{d (H^2e^\phi)}{dt}=2H\dot{H}e^\phi+H^2e^\phi\dot{\phi} \leq -H^3e^\phi+H^2e^\phi\dot{\phi}<0
				\label{hphi!}
			\end{equation*}
		which implies
			\begin{equation}
				H^2e^\phi|_{t=T_{\text{max}} }<\beta^2<\frac{1}{6}  \quad \implies \quad \frac{1}{6}  - 	H^2e^\phi|_{t=T_{\text{max}} }>0.
				\label{123}
			\end{equation}
			
				Evaluating the power identity at $t = T_{\text{max}}$, we obtain from Lemma \ref{proposition:3.2++} that  $\mathcal{P}(T_{\text{max}}) = 0$.  
Substituting $D_2(T_{\text{max}}) = 0$ into $\mathcal{P}(T_{\text{max}}) = 0$, with the help of \eqref{e:dp^2} and \eqref{123}, yields 
			\begin{align*}
				0=\mathcal{P}(T_{\text{max}} )=&\left. H \Biggl( \biggl(\frac{5}{6}-H^2e^\phi\biggr)\dot{\phi}^2+4H^3\dot{\phi}e^\phi+\frac{3}{2}H^6e^{2\phi} \Biggr) \right|_{t=T_{\text{max}} } \notag \\
				>&\left. H \Biggl( -6H^3e^\phi\biggl(\frac{1}{6}-H^2e^\phi\biggr)\dot{\phi}+\frac{3}{2}H^6e^{2\phi}  \Biggr) \right|_{t=T_{\text{max}} } >0 , 
			\end{align*}   
  where we have used $H>0$ guaranteed by Proposition \ref{proposition:5}, all the terms in the brackets are positive.  Together with $H(T_{\text{max}}) > 0$ , the entire expression is strictly positive, leading to the contradiction. 
		Hence, our assumption that $T_{\text{max}} < \mathcal{T}_+$ must be false, and we conclude that $T_{\text{max}} = \mathcal{T}_+$. This completes the proof. 
	\end{proof}

	\begin{proposition}
		\label{proposition:6}
		Under the initial conditions \eqref{eq:2.12!!}--\eqref{eq:2.11!!}, we have
		\begin{equation*}
			H(t) < \frac{1}{\frac{1}{2}t + \frac{1}{\beta}}, \quad \text{for all } t \in (0, \mathcal{T}_+).
		\end{equation*}
	\end{proposition}
	
	\begin{proof}
		The result then follows directly from Lemmas \ref{lemma:3.2!!!} and \ref{lemma:3.3!!!}.
	\end{proof}

\begin{lemma}\label{t:H2e}
		Under the initial conditions \eqref{eq:2.12!!}--\eqref{eq:2.11!!}, 	$H^2e^\phi <\beta^2$ for all $t \in (0, \mathcal{T}_+)$. 
\end{lemma}
\begin{proof}
	By Lemma \ref{lemma:3.3!!!}, we obtain $D_2(t)<0$ (i.e., $\dot H(t)<-\frac{1}{2}H^2(t)$) for all $t\in(0,\mathcal{T}_+)$. Then, noting $\dot\phi<0$ (by Proposition \ref{lemma:3.1}) and $H>0$ for $t\in(0,\mathcal{T}_+)$ (by Proposition \ref{proposition:5}), we derive
		\begin{equation*}
		\frac{d (H^2e^\phi)}{dt} =2H\dot{H}e^\phi+H^2e^\phi\dot{\phi}  < -H^3e^\phi+H^2e^\phi\dot{\phi} <0 
    	\end{equation*}
	for all $t\in(0,\mathcal{T}_+)$. Then $H^2e^\phi<\beta^2$ for all $t\in(0,\mathcal{T}_+)$. 
\end{proof}

	\begin{proposition}
		\label{proposition:phi2}
		Under the initial conditions \eqref{eq:2.12!!}--\eqref{eq:2.11!!}, we have
		\begin{equation*}
			-(12\beta^2+2\sqrt{6})\ln{\biggl(\frac{1}{2}\beta t+1\biggr)}<\phi(t)<-\ln{\biggl(1+\frac{3}{5}\Bigl(\beta^2-\frac{1}{(5t+\frac{1}{\beta})^2}\Bigr)\biggr)}.
		\end{equation*}
	\end{proposition}
	\begin{proof}
	Recall that Corollary \ref{dot phi} provides 
		\begin{equation*}
			-6H^3e^\phi-\sqrt{6}H<\dot{\phi}<-6H^3e^\phi,
		\end{equation*}
and by substituting the bounds of $H$ from Propositions \ref{proposition:5} and \ref{proposition:6}, together with the upper bound for 
 $H^2e^\phi$ given in Lemma \ref{t:H2e}, we obtain 
		\begin{equation*}
			-\frac{6\beta^2+\sqrt{6}}{\frac{1}{2}t+\frac{1}{\beta}}<\dot{\phi}<-6 \biggl(\frac{1}{5t+\frac{1}{\beta}}\biggr)^3e^\phi \quad \text{with}\quad \phi(0)=0 .
		\end{equation*}
		
Integrating both sides of the above inequalities yields the desired bounds for $\phi$.
	\end{proof}

\subsection{Bounds of $H$ for $t \in (\mathcal{T}_-,0)$}
	
	Now we proceed to establish the bounds in the region $t \in (\mathcal{T}_-,0)$. Define $D_3(t)$, $D_4(t)$, and $D_5(t)$  as  
	\begin{align}
	D_3:=&\dot{H}-3H^4e^\phi+H^2;
	\label{D_3} \\
	D_4:=&\dot{H}-\frac{12}{5}H^4e^\phi+3H^2;
	\label{D_4} \\
	D_5:=&\dot{H}-H^2+3e^{-\phi}.
	\label{D_5}
\end{align}

\begin{lemma}
	\label{1}
	Suppose $\beta>0$. 
	If  there is a finite constant $T\in (\mathcal{T}_-,0)$, such that $D_3(t)<0$ for all $t\in( T,0)$, then $H(t)>0$  for all $ t\in [T,0)$. 
\end{lemma}
\begin{proof}
Since $(\mathcal{T}_-,0)$ is the maximal backward interval of existence, Proposition~\ref{t:locext} ensures that $(H,\phi,\dot \phi) \in C^1((\mathcal{T}_-,0),\Rbb^3)$. From the assumption $D_3(t)<0$ for $t\in(T,0)$ and the expression \eqref{D_3}, we derive, for $t\in(T,0)$,
	\begin{equation*}
		\dot{H}<  3H^4e^\phi-H^2  \leq 3H^4e^\phi \quad \text{with } H(0)=\beta>0 . 
	\end{equation*}
Consider the comparison equation
\begin{equation*}
\dot{\underline{H}}= 3\underline{H}^4 e^\phi  \quad \text{with } \underline{H}(0)=\beta . 
\end{equation*}
Since $\phi(t)$ is known on $(T,0)$, direct integration yields
\begin{equation*}
	\underline{H}^{-3}(t)=\beta^{-3} + 9\int_t^0 e^{\phi(s)}\,ds \quad \implies \quad \underline{H}(t)=\Big(\beta^{-3} + 9\int_t^0 e^{\phi(s)}\,ds\Big)^{-1/3}  
\end{equation*}
for $t\in( T,0)$.

Note that $T\in (\mathcal{T}_-,0)$ implies 
\begin{equation}\label{e:inte}
	0<\int_T^0 e^{\phi(s)}\,ds \leq |T|\max_{t\in[T,0]} e^{\phi(t)}<+\infty . 
\end{equation} 
By the comparison theorem (Theorem~\ref{theorem:C}) and \eqref{e:inte}, it follows that
\begin{equation*}
	H(t)>\underline{H}(t)=\Big(\beta^{-3} + 9\int_t^0 e^{\phi(s)}\,ds\Big)^{-1/3} \geq \underline{H}(T)=\Big(\beta^{-3} + 9\int_T^0 e^{\phi(s)}\,ds\Big)^{-1/3} > 0 
\end{equation*}
for $t\in(T,0)$. 
Taking the limit as $t\rightarrow T+$ yields $H(T)\geq \underline{H}(T)>0$. 
This completes the proof. 
\end{proof}

		\begin{lemma}\label{t:Hphibd}
Suppose there exists a constant $T_0 \in(\mathcal{T}_-,0)$  such that  $D_3(t)<0$, $D_4(t)>0$  and $D_5(t)>0$ for all $t\in(T_0,0)$. Then $H$ satisfies
		\begin{align*}
			H(t)< & \min\left\{ \left(e^{6 \gamma  t} \left(\frac{1}{\beta ^2}-\frac{4 \beta }{5 \gamma }+\frac{2 \theta }{15 \gamma ^2}\right)-\frac{2 (-6 \beta  \gamma +\theta +6 \gamma  \theta  t)}{15 \gamma ^2}\right)^{-\frac{1}{2} },  \;\frac{\gamma}{-3\beta^3 t+1} \right\}   , \label{e:Hest1}\\
			H(t)> &   \left[ \beta^{-1/4} - \frac{\beta^{-1/4}}{4(m+1)} \left( 1 - (-3\beta^3 t+1)^{m+1} \right) \right]^{-4}>0   ,  
			\end{align*}
			and $\phi$ satisfies
			\begin{equation}\label{e:phiest1}
					\ln (-3\beta^3t+1)^2 < \phi < \frac{\sqrt{6} \gamma}{3\beta^3} \ln(-3\beta^3 t+1) +   \frac{2\beta^3}{4\gamma + \beta^3} \left[ (-3\beta^3 t+1)^{\frac{4\gamma + \beta^3}{\beta^3}} - 1 \right] 
			\end{equation} 
 		for $t\in (T_0,0)$, 	where
			\begin{equation*}
				\theta:= \frac{2\big(9\beta^{6}+\sqrt{9\beta^{6}+12}\,\beta^{3}\big)}{\big(\sqrt{9\beta^{6}+12}+3\beta^{3}\big)^{2}} , \quad 			\gamma =  \frac{3\beta^3+\sqrt{9\beta^6+12}}{2} \AND 	m =\frac{107 \sqrt{9 \beta ^6+12}+309 \beta ^3}{60 \beta ^3}  . 
			\end{equation*} 
	\end{lemma}
	\begin{proof}   
		To estimate 
	$\phi$ and $H$, we first derive a system of differential inequalities (Step 1), and then establish a hierarchy of estimates to handle them (Step 2). Solving these inequalities in a hierarchical manner yields the desired conclusion of the lemma.
		
		\underline{Step $1$: Derivations of differential inequalities.} 
		Since $D_3(t)<0$, $D_4(t)>0$  and $D_5(t)>0$ hold for all $t\in(T_0,0)$,  we obtain
		\begin{equation}\label{e:D3<0}
			\dot{H}<3 H^4e^\phi-H^2  , \quad 	\dot{H}>\frac{12}{5} H^4e^\phi- 3 H^2   \AND\dot{H}>H^2-3e^{-\phi}  
		\end{equation}
		for all $t\in(T_0,0)$. 
		In addition, we note that
			\begin{align}
			\frac{d (H^3e^\phi)}{dt}=& 3H^2\dot He^\phi+H^3e^\phi \dot{\phi},  \label{e:He1} \\
			\frac{d( He^\phi)}{dt} = & \dot H e^\phi +H e^\phi \dot\phi   ,  \label{e:He2}    \\ 	\frac{d(H^{\frac{11}{4}} e^\phi)}{dt} = & \frac{11}{4} H^{\frac{7}{4}} \dot H e^\phi+H^{\frac{11}{4}} e^\phi \dot\phi  .  \label{e:He3} 
		\end{align}
		
In the following analysis, we will frequently use the fact that $H>0$ for $t\in (\mathcal{T}_-,0)$, as stated in Lemma~\ref{1}, and this will no longer be mentioned explicitly.  

	Substituting the inequalities in \eqref{e:D3<0} into \eqref{e:He1}, respectively, and applying Corollary \ref{dot phi}, that is, $\dot{\phi}<-6H^3e^\phi$ and $\dot{\phi}>	-6H^3e^\phi-\sqrt{6}H$, we arrive at, 
	for all $t\in(T_0,0)$, 
		\begin{align}
	\frac{d (H^3e^\phi)}{dt} 
	< &  3 H^6e^{2\phi}-3H^4e^\phi   ,  \label{e:H4epineq1}\\
	\frac{d (H^3e^\phi)}{dt} 
	> & \frac{6}{5} H^6e^{2\phi}-12 H^4e^\phi . \label{e:H4epineq2} 
\end{align} 

Similarly, substituting the first inequality of \eqref{e:D3<0} into \eqref{e:He2}, and the second inequality of \eqref{e:D3<0} into \eqref{e:He3}, again using Corollary~\ref{dot phi}, we obtain, for all $t\in(T_0,0)$,  
		\begin{align}
	\frac{d( He^\phi)}{dt}  < & -3 H^4e^{2\phi}-H^2e^\phi  , \label{e:Heineq} \\
	\frac{d(H^{\frac{11}{4}} e^\phi)}{dt}   
	>& \frac{1}{10} H^{\frac{11}{4}} e^{\phi } \left(6 H^3 e^{\phi }-107H \right) . \label{e:H11/4eineq} 
\end{align}

\underline{Step $2$: Solving the system of differential inequalities.} 
For convenience, we introduce the quantities
		\begin{equation}\label{e:yzdef}
			y=H , \quad w=e^\phi, \quad z=H^3 e^\phi  , \quad  v=H e^\phi  \AND 
			p=H^{\frac{11}{4}} e^\phi   . 
		\end{equation}
With these new variables \eqref{e:yzdef}, the inequalities \eqref{e:D3<0}, \eqref{e:H4epineq1}–\eqref{e:H4epineq2}, \eqref{e:Heineq}–\eqref{e:H11/4eineq}, together with $\dot{\phi}<-6H^3e^\phi$ and $\dot{\phi}>	-6H^3e^\phi-\sqrt{6}H$, translate into the following system
		\begin{align}
			\dot y<&3 y^{\frac{5}{4}} p   -y^2  , \label{e:dy1} \\
			\dot y>& 
			\frac{12}{5}y^3v-3y^2
			, \label{e:dy2} \\
			\dot y> & y^2 -\frac{3}{w} , \label{e:dy3} \\
			\dot{z} < & 3z^2-3yz .  \label{e:dz1}\\
			\dot{z} > & \frac{6}{5}z^2-12yz .  \label{e:dz2}\\ 
			\dot{w}<& -6 z w.  \label{e:dw1} \\
			\dot w > &-6 z w -\sqrt{6} y w  ,  \label{e:dw2}   \\
			\dot v<& -3\frac{z^2}{y^2} -\frac{z}{y} ,   \label{e:dv1}   \\
			\dot p>&\frac{1}{10} p \left(6 z -107 y \right) ,  \label{e:dp1} 
		\end{align} 
		for all $t\in(T_0,0)$. 
These differential inequalities are supplemented with the initial data
		\begin{equation}\label{e:yzdata}
			y(0)=\beta, \quad z(0)=\beta^3, \quad w(0)=1 , \quad v(0)=\beta \AND p(0)=\beta^{\frac{11}{4}}  . 
		\end{equation}

In the following, we establish a hierarchy of estimates (refer to Fig.~\ref{f:flowchart}).

\underline{$(1)$ The lower bound for $z$.}
Since $yz=H^4 e^\phi>0$, \eqref{e:dz1} yields a Riccati-type inequality
		\begin{equation*}
			\dot{z} <   3z^2-3yz <3z^2. 
		\end{equation*}
With the initial condition $z(0) =\beta^3$ (recall \eqref{e:yzdata}), integration of the differential inequality gives, 
for all $t\in(T_0,0)$, 
		\begin{equation}\label{e:zlwb1}
			z> \frac{\beta^3}{-3\beta^3 t+1}  . 
		\end{equation}

\underline{$(2)$ The lower bound for $w$.}
Substituting the lower bound of 
$z$ from \eqref{e:zlwb1} into \eqref{e:dw1}, we obtain
\begin{equation*}
	\dot{w}<  -6 z w = -\frac{6 \beta^3}{-3\beta^3 t+1} w  \quad \text{with } w(0) =1 . 
\end{equation*} 
Integrating this differential inequality gives
\begin{equation}\label{e:wlbd}
	w=e^\phi>(-3\beta^3 t+1)^2 \quad \text{for all} \quad t \in (T_0, 0).
\end{equation}
Consequently, we obtain a lower bound for $\phi$, 
\begin{equation}\label{e:philwbd}
	\phi > \ln (-3\beta^3 t+1)^2 \quad \text{for all} \quad t \in (T_0, 0) , 
\end{equation}
which establishes the left-hand side of \eqref{e:phiest1}.

		\underline{$(3)$ The first upper bound for $y$.}		
	Substituting the lower bound of 
	$w$ in \eqref{e:wlbd} into \eqref{e:dy3}, we obtain
		\begin{equation*}
			\dot y>   y^2 -\frac{3}{w} >   y^2 -\frac{3}{(-3\beta^3 t+1)^2}  \quad \text{with } y(0)=\beta  . 
		\end{equation*} 
Consider the comparison equation
		\begin{equation}\label{e:bryeq}
				\dot{\bar{y} } =   \bar{y}^2 -\frac{3}{(-3\beta^3 t+1)^2}  \quad \text{with } \bar{y}(0)= \gamma \geq \beta=y(0) , 
		\end{equation}
where $\gamma$ is to be determined. 
A direct computation shows that
		\begin{equation}\label{e:theta}
			\bar{y}= \frac{\gamma}{-3\beta^3 t+1}    \quad \text{where }	\gamma :=  \frac{3\beta^3+\sqrt{9\beta^6+12}}{2} >\beta 
		\end{equation}
is a solution to \eqref{e:bryeq}.
Applying the comparison theorem (see, e.g., \cite[Theorem~$2.6.2$]{Hsu2013}), we conclude that, for all $t\in(T_0,0)$, 
		\begin{equation}\label{e:upbdy}
			y< \bar{y}= \frac{\gamma}{-3\beta^3 t+1}  <\gamma    . 
		\end{equation}

\underline{$(4)$ The lower bound for $v$.}	We now derive a lower bound for 
$v$ using \eqref{e:dv1}.
From \eqref{e:zlwb1} and \eqref{e:upbdy}, we have
		\begin{align*}
			-\frac{z}{y}  < &- \frac{\beta^3}{-3\beta^3 t+1}  \frac{1}{\frac{\gamma}{-3\beta^3 t+1} }  = - \frac{\beta^3}{ \gamma}    \\
			\intertext{and}
		 -\frac{z^2}{y^2}   < & - \left(\frac{\beta^3}{-3\beta^3 t+1}\right)^2  \frac{1}{(\frac{\gamma}{-3\beta^3 t+1})^2} =-\frac{\beta^6}{\gamma^2}. 
		\end{align*}
Therefore, combining \eqref{e:dv1} with \eqref{e:yzdata} yields
		\begin{equation}\label{e:dv2}
			\dot v<-\frac{3\beta^6}{\gamma^2}  - \frac{\beta^3}{ \gamma}   =-\frac{2 \left(9 \beta ^6+\sqrt{9 \beta ^6+12} \beta ^3\right)}{\left(\sqrt{9 \beta ^6+12}+3 \beta ^3\right)^2},  \quad \text{with }  v(0)=\beta . 
		\end{equation}
Integrating \eqref{e:dv2} then gives,  for all $t\in(T_0,0)$, 
		\begin{equation}\label{e:vlwbd}
			v(t)>\beta - \theta t   \quad \text{where }	\theta:= \frac{2\big(9\beta^{6}+\sqrt{9\beta^{6}+12}\,\beta^{3}\big)}{\big(\sqrt{9\beta^{6}+12}+3\beta^{3}\big)^{2}}  . 
		\end{equation}

		\underline{$(5)$ The second upper bound for $y$.}	 
	In this part, we derive an improved upper estimate for 
	$y$ by using \eqref{e:dy2}. This refined bound also shows that the differential inequality \eqref{e:dy2} guarantees that 
	$y$ remains bounded from above. 
	Recalling \eqref{e:dy2} and \eqref{e:yzdata}, we have
		\begin{equation}\label{e:dy2a}
			\dot y	>   \frac{12}{5}y^3v-3y^2   
			    \quad \text{with } y(0)=\beta  . 
		\end{equation}
		Multiplying both sides of \eqref{e:dy2a} by  	$y^{-4}$
		gives 
		\begin{equation*}
		-	\frac{1}{ 2 y }	\dot{q} 	 
		>   \frac{12}{5} \frac{v}{y} -3 q  , \quad \text{where }	q=y^{-2}  . 
		\end{equation*} 
Using \eqref{e:upbdy} and \eqref{e:vlwbd}, this implies
		\begin{equation*}
			\dot{q} 	 
			<   - \frac{24}{5} v + 6 y  q 	<   - \frac{24}{5} (\beta - \theta t ) + 6 \gamma   q  . 
		\end{equation*}
With the initial data \eqref{e:yzdata}, we obtain the differential inequality
		\begin{equation}\label{e:dq1}
			\dot{q} 	 -6 \gamma   q 
	 <   - \frac{24}{5} (\beta - \theta t )   \quad \text{with } q(0)=\beta^{-2} . 
		\end{equation}
This \eqref{e:dq1} is a \textit{linear} differential inequality, and solving it using an integrating factor yields
		\begin{equation*}
			q>e^{6 \gamma  t} \left(\frac{1}{\beta ^2}-\frac{4 \beta }{5 \gamma }+\frac{2 \theta }{15 \gamma ^2}\right)-\frac{2 (-6 \beta  \gamma +\theta +6 \gamma  \theta  t)}{15 \gamma ^2} . 
		\end{equation*}
Since $q=y^{-2}$, this gives the another upper bound
		\begin{equation*}
			y<\left(e^{6 \gamma  t} \left(\frac{1}{\beta ^2}-\frac{4 \beta }{5 \gamma }+\frac{2 \theta }{15 \gamma ^2}\right)-\frac{2 (-6 \beta  \gamma +\theta +6 \gamma  \theta  t)}{15 \gamma ^2}\right)^{-\frac{1}{2} }  . 
		\end{equation*}

Therefore, for all $t\in(T_0,0)$, 
		\begin{equation}\label{e:yupbd2}
			y<\min\left\{ \left(e^{6 \gamma  t} \left(\frac{1}{\beta ^2}-\frac{4 \beta }{5 \gamma }+\frac{2 \theta }{15 \gamma ^2}\right)-\frac{2 (-6 \beta  \gamma +\theta +6 \gamma  \theta  t)}{15 \gamma ^2}\right)^{-\frac{1}{2} },  \;\frac{\gamma}{-3\beta^3 t+1} \right\} , 
		\end{equation}
		where $\theta$ and $\gamma$ are defined in \eqref{e:theta} and \eqref{e:vlwbd}.

\underline{$(6)$ The upper bound for $z$.}		
From \eqref{e:dz2}, and using \eqref{e:upbdy}, we have
\begin{equation*}
\dot{z} >   \frac{6}{5}z^2-12yz >  -12yz > - \frac{12 \gamma z }{-3\beta^3 t+1}  \quad \text{with } z(0) = \beta^3 . 
\end{equation*} 
Integrating this inequality yields, for all $t\in(T_0,0)$, 
\begin{equation}\label{e:zupbd}
	z(t)<  \beta^3  \big(-3\beta^3 t+1\big)^{\frac{4\gamma}{\beta^3}} . 
\end{equation}

\underline{$(7)$ The upper bound for $w$.}		
From \eqref{e:dw2} and \eqref{e:yzdata}, together with   \eqref{e:upbdy} and \eqref{e:zupbd}, we obtain
\begin{equation*}
	\dot w >  -6 z w -\sqrt{6} y w> \left( -6 \beta^3  \big(-3\beta^3 t+1\big)^{\frac{4\gamma}{\beta^3}}   - \frac{\sqrt{6}\gamma}{-3\beta^3 t+1}\right)  w  \quad \text{with } w(0)=1. 
\end{equation*}
Integrating this differential inequality gives
\begin{equation*}
	w(t) < (-3\beta^3 t+1)^{\frac{\sqrt{6} \gamma}{3\beta^3}} \exp\left( \frac{2\beta^3}{4\gamma + \beta^3} \left[ (-3\beta^3 t+1)^{\frac{4\gamma + \beta^3}{\beta^3}} - 1 \right] \right) . 
\end{equation*}
Consequently,  for all $t\in(T_0,0)$, we further obtain the following upper bound for $	\phi $
\begin{equation}\label{e:phiupbd}
	\phi < \frac{\sqrt{6} \gamma}{3\beta^3} \ln(-3\beta^3 t+1) + \left( \frac{2\beta^3}{4\gamma + \beta^3} \left[ (-3\beta^3 t+1)^{\frac{4\gamma + \beta^3}{\beta^3}} - 1 \right] \right) .
\end{equation}

\underline{$(8)$ The upper bound for $p$.}		
Using \eqref{e:dp1} and \eqref{e:yzdata}, together with \eqref{e:zlwb1} and \eqref{e:upbdy}, we have
\begin{equation*}
	\dot p>\frac{1}{10} p \left(6 z -107 y \right)>\frac{1}{10} p \left(\frac{6 \beta^3}{-3\beta^3 t+1}
	 - \frac{107 \gamma}{-3\beta^3 t+1} \right) , \quad \text{with } p(0)=\beta^{\frac{11}{4}} . 
\end{equation*}
Integrating this inequality gives,  for all $t\in(T_0,0)$,  
\begin{equation}\label{e:pupbd}
	p<\beta ^{\frac{11}{4}} \left(-3\beta^3 t+1\right)^{\frac{107 \gamma }{30 \beta ^3}-\frac{1}{5}}=\beta ^{\frac{11}{4}} \left(-3\beta^3 t+1\right)^{\frac{107 \left(\sqrt{9 \beta ^6+12}+3 \beta ^3\right)}{60 \beta ^3}-\frac{1}{5}} .
\end{equation}

\underline{$(9)$ The lower bound for $y$.}		
From \eqref{e:dy1} and \eqref{e:yzdata}, together with \eqref{e:pupbd}, we have
\begin{equation*}
	\dot y< 3 y^{\frac{5}{4}} p   -y^2  < 3 y^{\frac{5}{4}} \beta ^{\frac{11}{4}} \left(-3\beta^3 t+1\right)^{\frac{107 \left(\sqrt{9 \beta ^6+12}+3 \beta ^3\right)}{60 \beta ^3}-\frac{1}{5}}  \quad \text{with } y(0) = \beta . 
\end{equation*}
Integrating this inequality yields, for all $t\in(T_0,0)$,  
\begin{equation}\label{e:ylwbd}
	y(t) > \left[ \beta^{-1/4} - \frac{\beta^{-1/4}}{4(m+1)} \left( 1 - (-3\beta^3 t+1)^{m+1} \right) \right]^{-4}>0  , 
\end{equation}
where
\begin{equation*}
	m =\frac{107 \gamma }{30 \beta ^3}-\frac{1}{5}= \frac{107 \left( \sqrt{9\beta^6 + 12} + 3\beta^3 \right)}{60 \beta^3} - \frac{1}{5} =\frac{107 \sqrt{9 \beta ^6+12}+309 \beta ^3}{60 \beta ^3} . 
\end{equation*}
 
Combining \eqref{e:philwbd}, \eqref{e:yupbd2}, \eqref{e:phiupbd}, and \eqref{e:ylwbd} completes the proof.
	\end{proof}

	\begin{lemma}
		\label{2}
		Suppose $\beta>0$. If $D_3(0)<0$, then $D_3(t) < 0$ for all $t \in (\mathcal{T}_-,0)$.
	\end{lemma}
	\begin{proof}
		Since $D_3(0) < 0$ and $D_3(t)$ is continuous (by the definition \eqref{D_3} and Proposition \ref{t:locext}), there exists a constant $T \in [\mathcal{T}_-, 0)$ such that $D_3(t) < 0$ for all $t \in (T, 0)$. Define  
		\begin{equation}\label{e:Tmin!+}
			T_{\text{min}} := \inf\{T \in [\mathcal{T}_-, 0) \mid D_3(t) < 0 \text{ for all } t \in (T, 0)\}.
		\end{equation}
		
		We prove $T_{\text{min}} = \mathcal{T}_-$ by contradiction. Suppose $T_{\text{min}} > \mathcal{T}_-\geq -\infty$. By the continuity of $D_3$, we must have $D_3(T_{\text{min}}) = 0$; otherwise, this  contradicts the definition of $T_{\text{min}}$ in \eqref{e:Tmin!+}.
		
		Note $H(t)>0$   for all $ t\in [T_{\text{min}},0)$ from Lemma \ref{1}. Thus, $H(T_{\text{min}})>0$.   Evaluating the power identity at $t = T_{\text{min}}$.  By Lemma \ref{proposition:3.2++}, we have $\mathcal{P}(T_{\text{min}}) = 0$, and taking $D_3(T_{\text{min}}) = 0$ into account yields 
		\begin{align}
			0=\mathcal{P}(T_{\text{min}} )=&\left. H \Biggl( \frac{2}{3} \dot{\phi}^2+4H^3\dot{\phi}e^\phi+9H^8 e^{3\phi}    \Biggr) \right|_{t=T_{\text{min}} }  \notag  \\
			\overset{\text{Cor. } \ref{dot phi}}{>}&\left. H \Biggl( \frac{2}{3} \dot{\phi}^2- \frac{2}{3} \dot{\phi}^2+9H^8 e^{3\phi}   \Biggr) \right|_{t=T_{\text{min}} }  \notag  \\
			> & 9H^9e^{3\phi}  \big|_{t=T_{\text{min}} } >0 . \label{e:0>0!!!}
		\end{align}  
This leads to the contradiction \eqref{e:0>0!!!}. 
		Hence, our assumption that $T_{\text{min}} > \mathcal{T}_-$ must be false, and we conclude that $T_{\text{min}} = \mathcal{T}_-$. This completes the proof. 
	\end{proof}

	\begin{lemma}
		\label{3}
		Suppose $\beta>0$.  If $D_4(0)>0$ and $D_3(0)<0$, then $D_4(t) > 0$ for all $t \in (\mathcal{T}_-,0)$.
	\end{lemma}
	\begin{proof}
Since $D_4(0) > 0$ and $D_4(t)$ is continuous (by Definition \eqref{D_4} and Proposition \ref{t:locext}), there exists a constant $T \in [\mathcal{T}_-, 0)$ such that $D_4(t) > 0$ for all $t \in (T, 0)$. Define  
		\begin{equation}\label{e:Tmin+}
			T_{\text{min}} := \inf\{T \in [\mathcal{T}_-, 0) \mid D_4(t) > 0 \text{ for all } t \in (T, 0)\}.
		\end{equation}
		
		We prove $T_{\text{min}} = \mathcal{T}_-$ by contradiction. Suppose $T_{\text{min}} > \mathcal{T}_-$. By the continuity of $D_4$, we must have $D_4(T_{\text{min}}) = 0$; otherwise, this   contradicts the definition of $T_{\text{min}}$ in \eqref{e:Tmin+}.

		Since $D_3(0)<0$ and $\beta>0$, by Lemma \ref{2}, we obtain $D_3(t) < 0$ for all $t \in (\mathcal{T}_-,0)$. We further use Lemma \ref{1}, then  $H(t)>0$ for all $ t\in [T_{\text{min}},0)$.  Thus, $H(T_{\text{min}})>0$.   
		
		Evaluating the power identity at $t = T_{\text{min}}$. By Lemma \ref{proposition:3.2++}, we have $\mathcal{P}(T_{\text{min}}) = 0$, Corollary \ref{dot phi} gives rise to $\dot{\phi}<-6H^3e^\phi$, thus $\dot{\phi}^2>36H^6e^{2\phi}$, and taking $D_4(T_{\text{min}}) = 0$ into account yields 
		\begin{align}
			0=\mathcal{P}(T_{\text{min}} )=&\left. H \Biggl( -\frac{1}{5}H^2 e^\phi\dot{\phi}^2+4H^3\dot{\phi}e^\phi+3H^6e^{2\phi}\biggl(\frac{12}{5}H^2e^\phi-2\biggr) \Biggr) \right|_{t=T_{\text{min}} } \notag \\
			<&\left. H (4H^3\dot{\phi}e^\phi-6H^6e^{2\phi} )  \right|_{t=T_{\text{min}} }<0
			\label{e:0>0!!}
		\end{align}  
		Note that every term in the bracket is negative and $H(T_{\text{min}})>0$, leading to the contradiction \eqref{e:0>0!!}. 
		Hence, our assumption that $T_{\text{min}} > \mathcal{T}_-$ must be false, and we conclude that $T_{\text{min}} = \mathcal{T}_-$. This completes the proof. 
	\end{proof}

	\begin{lemma}
		\label{3.5}
		Suppose $\beta>0$. If $D_5(0)>0$ and $D_3(0)<0$, then $D_5(t) > 0$ for all $t \in (\mathcal{T}_-,0)$.
	\end{lemma}
	\begin{proof}
		Since $D_5(0) > 0$ and $D_5(t)$ is continuous (by Definition \eqref{D_5} and Proposition \ref{t:locext}), there exists a constant $T \in [\mathcal{T}_-, 0)$ such that $D_5(t) > 0$ for all $t \in (T, 0)$. Define  
		\begin{equation}\label{e:Tmin+!}
			T_{\text{min}} := \inf\{T \in [\mathcal{T}_-, 0) \mid D_5(t) > 0 \text{ for all } t \in (T, 0)\}.
		\end{equation}
		
		We prove $T_{\text{min}} = \mathcal{T}_-$ by contradiction. Suppose $T_{\text{min}} > \mathcal{T}_-$. By the continuity of $D_5$, we must have $D_5(T_{\text{min}}) = 0$; otherwise, this   contradict the definition of $T_{\text{min}}$ in \eqref{e:Tmin+!}.

		Since $D_3(0)<0$ and $\beta>0$, by Lemma \ref{2}, we obtain $D_3(t) < 0$ for all $t \in (\mathcal{T}_-,0)$. We further use Lemma \ref{1}, then  $H(t)>0$ for all $ t\in [T_{\text{min}},0)$.  Thus, $H(T_{\text{min}})>0$.

		Evaluating the power identity at $t = T_{\text{min}}$. By Lemma \ref{proposition:3.2++}, we have $\mathcal{P}(T_{\text{min}}) = 0$ and Corollary \ref{dot phi} gives rise to $\dot{\phi}<-6H^3e^\phi$.  Noting $H^2e^\phi+\frac{1}{H^2e^\phi}>2$ and taking $D_5(T_{\text{min}}) = 0$ into account yields 
		\begin{align}
			0=\mathcal{P}(T_{\text{min}} )=&\left. H \Biggl( \biggl(\frac{4}{3}-H^2e^\phi-\frac{1}{H^2e^\phi}\biggr)\dot{\phi}^2+4H^3\dot{\phi}e^\phi+3H^6e^{2\phi}\biggl(2-\frac{3}{H^2e^\phi}\biggr)\Biggr) \right|_{t=T_{\text{min}} } \notag \\
			<&\left. H \Biggl(-\frac{2}{3}\dot{\phi}^2-18 H^6e^{2\phi}-9H^4e^\phi\Biggr)  \right|_{t=T_{\text{min}} }<0 . 
			\label{e:0>0!!!!}
		\end{align}  
	Note that every term in the bracket is negative and $H(T_{\text{min}})>0$, leading to the contradiction \eqref{e:0>0!!!!}. 
		Hence, our assumption that $T_{\text{min}} > \mathcal{T}_-$ must be false, and we conclude that $T_{\text{min}} = \mathcal{T}_-$. This completes the proof. 
	\end{proof}

	\begin{lemma} \label{4}
		Under the initial conditions \eqref{eq:2.12!!}--\eqref{eq:2.11!!}, $D_3(0) < 0$ and $D_4(0) > 0$, $D_5(0) > 0$.
	\end{lemma}
	
	\begin{proof}
		This proof can be derived by direct computation as the proof of Lemma \ref{t:dataB!}, we omit the detail here.
	\end{proof}

	\begin{proposition}\label{proposition:8}
		Under the initial conditions \eqref{eq:2.12!!}--\eqref{eq:2.11!!},  $H$ satisfies
		\begin{align*}
			H(t)< & \min\left\{ \left(e^{6 \gamma  t} \left(\frac{1}{\beta ^2}-\frac{4 \beta }{5 \gamma }+\frac{2 \theta }{15 \gamma ^2}\right)-\frac{2 (-6 \beta  \gamma +\theta +6 \gamma  \theta  t)}{15 \gamma ^2}\right)^{-\frac{1}{2} },  \;\frac{\gamma}{-3\beta^3 t+1} \right\}   , \\
			H(t)> &   \left[ \beta^{-1/4} - \frac{\beta^{-1/4}}{4(m+1)} \left( 1 - (-3\beta^3 t+1)^{m+1} \right) \right]^{-4}>0    , 
		\end{align*}
		and $\phi$ satisfies
		\begin{equation*} 
			\ln (-3\beta^3t+1)^2 < \phi < \frac{\sqrt{6} \gamma}{3\beta^3} \ln(-3\beta^3 t+1) +   \frac{2\beta^3}{4\gamma + \beta^3} \left[ (-3\beta^3 t+1)^{\frac{4\gamma + \beta^3}{\beta^3}} - 1 \right] 
		\end{equation*} 
		for $t\in (\mathcal{T}_-,0)$, 	where
		\begin{equation*}
			\theta:= \frac{2\big(9\beta^{6}+\sqrt{9\beta^{6}+12}\,\beta^{3}\big)}{\big(\sqrt{9\beta^{6}+12}+3\beta^{3}\big)^{2}} , \quad 			\gamma =  \frac{3\beta^3+\sqrt{9\beta^6+12}}{2} \AND 	m =\frac{107 \sqrt{9 \beta ^6+12}+309 \beta ^3}{60 \beta ^3}  . 
		\end{equation*} 
	\end{proposition}
	\begin{proof}
		This proof follows directly from Lemmas \ref{t:Hphibd}--\ref{4}.
	\end{proof}

	\section{Proof of the Main Theorem}\label{sec:4}
    Now we are in a position to prove the main theorem \ref{theorem:2.2}. 
    \begin{proof}[Proof of Theorem \ref{theorem:2.2}] 
	The argument proceeds in two steps. First, we recall the local estimates established in \S\ref{sec:3}. Then, we extend these estimates to obtain the desired global result.

   \underline{The local bounds:} In Proposition \ref{t:locext}, we established the local existence of solutions to the system \eqref{eq:2.5}–\eqref{eq:2.6}, together with bounds for $H$ and $\phi$ on the interval $(\mathcal{T}_-,\mathcal{T}_+)$ corresponding to the initial data \eqref{initial-data!!}–\eqref{eq:2.11!!}. 
For convenience, we recall and collect here all estimates for $H$ and $\phi$  on $t \in (0,\mathcal{T}_+)$ given in Propositions \ref{proposition:5}, \ref{proposition:6} and  \ref{proposition:phi2} here, 
			\begin{equation*}
				\frac{1}{5t+\frac{1}{\beta}} < H(t) < \frac{1}{\frac{1}{2}t+\frac{1}{\beta}}, 
			\end{equation*}
			\begin{equation*}
						-(12\beta^2+2\sqrt{6})\ln{\biggl(\frac{1}{2}\beta t+1\biggr)}<\phi(t)<-\ln{\left(1+\frac{3}{5}\left(\beta^2-\frac{1}{(5t+\frac{1}{\beta})^2}\right)\right)}.
			\end{equation*} 
In addition, the corresponding bounds for  $H$ and $\phi$ on $t\in(\mathcal{T}_-,0)$ are given in Proposition \ref{proposition:8}.

		\underline{Extensions of solutions:} Let us now prove $\mathcal{T}_- = -\infty$ and $\mathcal{T}_+ = +\infty$ by contradiction. 	
		We first assume, for contradiction, that $\mathcal{T}_+ < +\infty$, and focus on the solution on $(0,\mathcal{T}_+)$. 
		Using the bounds for $H$ and $\phi$ established above, together with Corollary \ref{dot phi}, we obtain
		\begin{equation*}
			|\dot{\phi}|<6H^3e^\phi+\sqrt{6}H<(6\beta^2+\sqrt{6})\beta. 
		\end{equation*}
		Thus there exists a constant
		\begin{equation*}
			R>\max\left\{   \beta,       (12\beta^2+2\sqrt{6})\ln{\left(\frac{1}{2}\beta \mathcal{T}_++1\right)},   (6\beta^2+\sqrt{6})\beta    \right\}>0, 
		\end{equation*}
		such that $\mathcal{U}:=\p{H,\phi,\Phi}^T\in B_R(0) \subset \Rbb^3$ for all $t\in(0,\mathcal{T}_+)$, where $B_R(0)$ denotes the open ball centered at the origin with radius $R$.

		Recall that the system \eqref{e:locsys} can be written as the ODE:
		\begin{equation*}
			\frac{d}{dt} \mathcal{U}= \mathcal{F}(\mathcal{U})
		\end{equation*}
		where  
		\begin{equation*}
			\mathcal{F}(\mathcal{U}):= \p{ F_1(H,\phi,\Phi) \\
				\Phi \\ F_2(H,\phi,\Phi)  }
		\end{equation*}
		and $F_1$ and $F_2$ are defined in \eqref{e:F1} and \eqref{e:F2}, respectively. Since $\mathcal{F}\in C^1(\overline{B_R(0)}, \Rbb^3)$, it is Lipschitz continuous and bounded on $\overline{B_R(0)}$. 
		Considering Corollary~\ref{t:contthm2}, the solution $\mathcal{U}$ can therefore be continued to the right passing through the point $\mathcal{T}_+$. This contradicts the assumption that $\mathcal{T}_+$ is the maximal time of existence. Hence, we conclude that $\mathcal{T}_+ = +\infty$. 
		
A similar argument applies backward in time.
By Proposition \ref{proposition:8}, $H$ is bounded above by a constant $\gamma$ on  $(\mathcal{T}_-,0)$. In addition, $\phi$ has a constant upper bound
\begin{equation*}
	R_0:=\frac{\sqrt{6} \gamma}{3\beta^3} \ln(-3\beta^3 \mathcal{T}_-+1) +   \frac{2\beta^3}{4\gamma + \beta^3} \left[ (-3\beta^3 \mathcal{T}_-+1)^{\frac{4\gamma + \beta^3}{\beta^3}} - 1 \right] 
\end{equation*}
on this interval as well. Choose
		\begin{equation*}
			R>\max\left\{\gamma,    R_0, 6 \gamma^3 e^{R_0}+\sqrt{6}\gamma \right\}>0,
		\end{equation*}
and applying the same continuation argument yields $\mathcal{T}_-=-\infty$.     
This completes the proof.  
	\end{proof}

	\appendix

\section{Basic ode theorems}
	\label{app:A}
	This appendix presents fundamental existence and comparison theorems employed in the paper. The proofs are omitted here as they can be found in standard ODE textbooks such as \cite{Hsu2013,you1982}. 
	\begin{theorem}[Picard-Lindelöf Theorem]
		\label{theorem:PL}
		Consider a closed domain $ D \subseteq \mathbb{R} \times \mathbb{R}^n $ containing the point $ (t_0, y_0) $, and let $ f: D \to \mathbb{R}^n $ be a continuous function that is Lipschitz continuous in $y$ with constant $L$. Then the initial value problem
		\begin{equation}\label{e:ode1}
			\frac{dy}{dt} = f(t, y), \quad y(t_0) = y_0
		\end{equation}
		admits a unique $C^1$ solution $ y(t) $ defined on some interval $ [t_0 - h, t_0 + h] $ with $ h > 0 $.
	\end{theorem} 
	\begin{theorem}[Comparison Theorem]
		\label{theorem:C}
		Let $ f(t,x) $ and $ F(t,x) $ be continuous scalar functions defined on a planar region $ D $, satisfying
		\begin{equation*}
			f(t,x) < F(t,x), \quad (t,x) \in D. 
		\end{equation*}
		If $ x = \varphi(t) $ and $ x = \Phi(t) $ are solutions to the differential equations
		\begin{equation*}
			x' = f(t,x) \quad \text{and} \quad x' = F(t,x), 
		\end{equation*}
		respectively, that both pass through the point $ (\tau, \xi) \in D $, then the following holds:
		\begin{enumerate}[leftmargin=*]
			\item[(1)] $ \varphi(t) < \Phi(t) $ for $ t > \tau $ within their common interval of existence;
			\item[(2)] $ \varphi(t) > \Phi(t) $ for $ t < \tau $ within their common interval of existence.
		\end{enumerate}
	\end{theorem}
	
	\begin{theorem}[Continuation of solutions]\label{t:contthm1}
		Let $f\in C(D)$ and assume $|f(t,y)|\leq M$ for some constant $M>0$ and all $(t,y) \in D$. If $\phi$ is a solution of \eqref{e:ode1} on the interval $J=(a,b)$, then
		
		$(1)$ the limits $\lim_{t\rightarrow a+} \phi(t)=\phi(a)$ and $\lim_{t\rightarrow b-} \phi(t)=\phi(b)$ exist and are finite;
		
		$(2)$ if $(b,\phi(b)) \in D$, then the solution $\phi$ can be extended to the right beyond $t=b$.
	\end{theorem}
	
	\begin{corollary}[Continuation principle]\label{t:contthm2}
		Let $f\in C(D)$. Suppose $\phi$ is a solution of \eqref{e:ode1} on the interval $J=(a,b)$, and there exists a finite constant $M>0$ such that for every $t\in (a,b)$,
		\begin{equation*}
			|f(t,\phi(t))|\leq M<+\infty,
		\end{equation*}
		then the solution $\phi$ can be continued to the right beyond $t=b$.
	\end{corollary}

	\section{Useful inequalities}
	\label{app:B}
 
	\begin{lemma}\label{t:bscineq}
		For any $x>0$,  the following inequalities are valid:
		\begin{equation}
			\sqrt{1+x^2}<1+x,
			\label{eq:B2}
		\end{equation}
		\begin{equation}
			\sqrt{1+x^2}-x>\frac{1}{1+2x},
			\label{eq:B1}
		\end{equation}
		\begin{equation}
			\sqrt{1+x^2}>1.
			\label{eq:B3}
		\end{equation}
	\end{lemma}

	\section*{Acknowledgements}
	
	C.L. is partially supported by  NSFC (Grant No. $12571234$).

	\bigskip
	
	\textbf{Data Availability} Data sharing is not applicable to this article as no datasets were
	generated or analysed during the current study.
	
	\bigskip
	
	\textbf{Declarations}
	
	\bigskip
	
	\textbf{Conflict of interest} The authors declare that they have no conflict of interest.

	\bibliographystyle{amsplain}
	\bibliography{ref}

@misc{he2025proofssingularityfreesolutionsscalarization,
      title={Proofs of singularity-free solutions and scalarization in nonlinear Einstein-scalar-Gauss-Bonnet cosmology}, 
      author={Chihang He and Chao Liu and Jinhua Wang},
      year={2025},
      eprint={2507.15304},
      archivePrefix={arXiv},
      primaryClass={math.AP},
      url={https://arxiv.org/abs/2507.15304}, 
}

@Article{Alexeyev2000,
doi = {10.1088/0264-9381/17/11/306},
url = {https://dx.doi.org/10.1088/0264-9381/17/11/306},
year = {2000},
month = {jun},
publisher = {},
volume = {17},
number = {11},
pages = {2243},
author = {S O Alexeyev and A V Toporensky and V O Ustiansky},
title = {Non-singular
cosmological models in string gravity with constant dilaton and
second-order curvature corrections},
journal = {Classical and Quantum Gravity}
}

@Article{Hikmawan2016,
  author    = {Getbogi Hikmawan and Jiro Soda and Agus Suroso and Freddy Permana Zen},
  journal   = {Physical Review D},
  title     = {{Comment on ''Gauss-Bonnet inflation''}},
  year      = {2016},
  month     = {Mar},
  pages     = {068301},
  volume    = {93},
  doi       = {10.1103/PhysRevD.93.068301},
  issue     = {6},
  numpages  = {5},
  publisher = {American Physical Society},
  url       = {https://link.aps.org/doi/10.1103/PhysRevD.93.068301},
}

@Article{Kanti2015,
  author    = {Panagiota Kanti and Radouane Gannouji and Naresh Dadhich},
  journal   = {Physical Review D},
  title     = {{Early-time cosmological solutions in Einstein-scalar-Gauss-Bonnet theory}},
  year      = {2015},
  month     = {Oct},
  pages     = {083524},
  volume    = {92},
  doi       = {10.1103/PhysRevD.92.083524},
  issue     = {8},
  numpages  = {15},
  publisher = {American Physical Society},
  url       = {https://link.aps.org/doi/10.1103/PhysRevD.92.083524},
}

@Article{Kanti2015a,
  author    = {Panagiota Kanti and Radouane Gannouji and Naresh Dadhich},
  journal   = {Physical Review D},
  title     = {{Gauss-Bonnet inflation}},
  year      = {2015},
  month     = {Aug},
  pages     = {041302},
  volume    = {92},
  doi       = {10.1103/PhysRevD.92.041302},
  issue     = {4},
  numpages  = {5},
  publisher = {American Physical Society},
  url       = {https://link.aps.org/doi/10.1103/PhysRevD.92.041302},
}

@Article{Nojiri2005,
  author    = {Shin'ichi Nojiri and Sergei D. Odintsov and Misao Sasaki},
  journal   = {Physical Review D},
  title     = {{Gauss-Bonnet dark energy}},
  year      = {2005},
  month     = {Jun},
  pages     = {123509},
  volume    = {71},
  doi       = {10.1103/PhysRevD.71.123509},
  issue     = {12},
  numpages  = {7},
  publisher = {American Physical Society},
  url       = {https://link.aps.org/doi/10.1103/PhysRevD.71.123509},
}

@Article{Nojiri2006,
  author    = {Shin’ichi Nojiri and Sergei D. Odintsov and M. Sami},
  journal   = {Physical Review D},
  title     = {{Dark energy cosmology from higher-order, string-inspired gravity, and its reconstruction}},
  year      = {2006},
  month     = {Aug},
  pages     = {046004},
  volume    = {74},
  doi       = {10.1103/PhysRevD.74.046004},
  issue     = {4},
  numpages  = {14},
  publisher = {American Physical Society},
  url       = {https://link.aps.org/doi/10.1103/PhysRevD.74.046004},
}

@Article{Kanti1999,
  author    = {Panagiota Kanti and John Rizos and Kyriakos Tamvakis},
  journal   = {Physical Review D},
  title     = {{Singularity-free cosmological solutions in quadratic gravity}},
  year      = {1999},
  month     = {Mar},
  pages     = {083512},
  volume    = {59},
  doi       = {10.1103/PhysRevD.59.083512},
  issue     = {8},
  numpages  = {12},
  publisher = {American Physical Society},
  url       = {https://link.aps.org/doi/10.1103/PhysRevD.59.083512},
}

@Article{Gross1987,
  author   = {David J. Gross and John H. Sloan},
  journal  = {Nuclear Physics B},
  title    = {The quartic effective action for the heterotic string},
  year     = {1987},
  issn     = {0550-3213},
  pages    = {41-89},
  volume   = {291},
  doi      = {https://doi.org/10.1016/0550-3213(87)90465-2},
  url      = {https://www.sciencedirect.com/science/article/pii/0550321387904652},
}

@Article{Clifton2012,
  author   = {Timothy Clifton and Pedro G. Ferreira and Antonio Padilla and Constantinos Skordis},
  journal  = {Physics Reports},
  title    = {Modified gravity and cosmology},
  year     = {2012},
  issn     = {0370-1573},
  note     = {Modified Gravity and Cosmology},
  number   = {1},
  pages    = {1-189},
  volume   = {513},
  doi      = {https://doi.org/10.1016/j.physrep.2012.01.001},
  url      = {https://www.sciencedirect.com/science/article/pii/S0370157312000105},
}

@Article{Rizos1994,
  author   = {John Rizos and Kyriakos Tamvakis},
  journal  = {Physics Letters B},
  title    = {{On the existence of singularity-free solutions in quadratic gravity}},
  year     = {1994},
  issn     = {0370-2693},
  number   = {1},
  pages    = {57-61},
  volume   = {326},
  doi      = {https://doi.org/10.1016/0370-2693(94)91192-4},
  url      = {https://www.sciencedirect.com/science/article/pii/0370269394911924},
}

@Article{Kawai1998,
  author   = {Shinsuke Kawai and Masa-aki Sakagami and Jiro Soda},
  journal  = {Physics Letters B},
  title    = {Instability of 1-loop superstring cosmology},
  year     = {1998},
  issn     = {0370-2693},
  number   = {3},
  pages    = {284-290},
  volume   = {437},
  doi      = {https://doi.org/10.1016/S0370-2693(98)00925-3},
  url      = {https://www.sciencedirect.com/science/article/pii/S0370269398009253},
}

@InProceedings{Kawai1999,
  author        = {Shinsuke Kawai and Masa-aki Sakagami and Jiro Soda},
  booktitle     = {Proceedings, 7th Workshop on General Relativity and Gravitation (JGRG7) : Kyoto, Japan, October 27-30, 1997},
  title         = {Perturbative analysis of non-singular cosmological model},
  year          = {1999},
  archiveprefix = {arXiv},
  eprint        = {gr-qc/9901065},
  primaryclass  = {gr-qc},
  url           = {https://arxiv.org/abs/gr-qc/9901065},
}

@MastersThesis{Sberna2017a,
  author        = {Laura Sberna},
  school        = {University of Rome},
  title         = {{Early-universe cosmology in Einstein-scalar-Gauss-Bonnet gravity}},
  year          = {2017},
  archiveprefix = {arXiv},
  eprint        = {1708.01150},
  primaryclass  = {gr-qc},
  url           = {https://arxiv.org/abs/1708.01150},
}

@Book{Hsu2013,
  author    = {Sze-Bi Hsu},
  publisher = {World scientific},
  title     = {Ordinary Differential Equations with Applications},
  year      = {2013},
  month     = {jan},
  doi       = {10.1142/8744},
}

@Book{ChoquetBruhat2009,
  author    = {Yvonne Choquet-Bruhat},
  publisher = {Oxford University Press},
  title     = {{General relativity and the Einstein equations}},
  year      = {2009},
}

@book{Ringstroem2009,
	Author = {Ringstr{\"o}m, Hans},
	Isbn = {9783037190531},
	Lccn = {2009281497},
	Publisher = {European Mathematical Society},
	Series = {ESI lectures in mathematics and physics},
	Title = {The {Cauchy} Problem in General Relativity},
	Url = {https://books.google.com.au/books?id=Bn\_cC7QwQ0MC},
	Year = {2009}}

@Article{Liu2022,
  author    = {Chao Liu and Yiqing Shi},
  journal   = {Physical Review D},
  title     = {Rigorous proof of the slightly nonlinear {Jeans} instability in the expanding {Newtonian} universe},
  year      = {2022},
  month     = {feb},
  number    = {4},
  pages     = {043519},
  volume    = {105},
  doi       = {10.1103/physrevd.105.043519},
  publisher = {American Physical Society ({APS})},
}

@Article{Liu2022b,
  author    = {Chao Liu},
  journal   = {Mathematische Annalen},
  title     = {Blowups for a class of second order nonlinear hyperbolic equations: a reduced model of nonlinear Jeans instability},
  year      = {2025},
  date      = {2025/09/01},
  volume    = {393},
  number    = {1},
  pages     = {317--363},
  issn      = {1432-1807},
  doi       = {10.1007/s00208-025-03260-0},
  url       = {},
}

@Article{Liu2023a,
  author  = {Chao Liu},
  journal = {Physical Review D},
  title   = {Fully nonlinear gravitational instabilities for expanding {Newtonian} universes with inhomogeneous pressure and entropy: beyond the {T}olman's solution},
  year    = {2023},
  month   = jun,
  number  = {12},
  pages   = {123534},
  volume  = {107},
  doi     = {10.1103/PhysRevD.107.123534},
  eprint  = {https://arxiv.org/abs/2210.04657#},
  url     = {https://journals.aps.org/prd/abstract/10.1103/PhysRevD.107.123534},
}

@Article{Liu2023b,
  author        = {Chao Liu},
  journal       = {arXiv:2305.13211},
  title         = {{Fully nonlinear gravitational instabilities for expanding spherical symmetric Newtonian universes with inhomogeneous density and pressure}},
  year          = {2023},
  month         = may,
  archiveprefix = {arXiv},
  eprint        = {2305.13211},
  primaryclass  = {math-ph},
}

@Article{Liu2024,
  author        = {Chao Liu},
  journal       = {arXiv:2409.02516},
  title         = {{The emergence of nonlinear Jeans-type instabilities for quasilinear wave equations}},
  year          = {2024},
  month         = sep,
  archiveprefix = {arXiv},
  copyright     = {arXiv.org perpetual, non-exclusive license},
  doi           = {10.48550/ARXIV.2409.02516},
  eprint        = {2409.02516},
  primaryclass  = {math.AP},
  publisher     = {arXiv},
}

@Book{you1982,
  author    = {You, Bingli},
  publisher = {Higher Education Press},
  title     = {{Supplementary Course in Ordinary Differential Equations}},
  year      = {1981},
  address   = {Beijing},
  isbn      = {9781124095455},
  language  = {in Chinese},
}

@Article{Easther1996,
  author    = {Richard Easther and Kei-ichi Maeda},
  journal   = {Physical Review D},
  title     = {One-loop superstring cosmology and the nonsingular universe},
  year      = {1996},
  issn      = {1089-4918},
  month     = dec,
  number    = {12},
  pages     = {7252--7260},
  volume    = {54},
  doi       = {10.1103/physrevd.54.7252},
  publisher = {American Physical Society (APS)},
}

@Article{NOVELLO2008,
  author    = {M. Novello and S.E. Perez Bergliaffa},
  journal   = {Physics Reports},
  title     = {Bouncing cosmologies},
  year      = {2008},
  issn      = {0370-1573},
  month     = jul,
  number    = {4},
  pages     = {127--213},
  volume    = {463},
  doi       = {10.1016/j.physrep.2008.04.006},
  publisher = {Elsevier BV},
}

@Article{KalyanaRama1997,
  author    = {Kalyana Rama, S.},
  journal   = {Physical Review Letters},
  title     = {Singularity-Free Homogeneous Isotropic Universe in Graviton-Dilaton Models},
  year      = {1997},
  issn      = {1079-7114},
  month     = mar,
  number    = {9},
  pages     = {1620--1623},
  volume    = {78},
  doi       = {10.1103/physrevlett.78.1620},
  publisher = {American Physical Society (APS)},
}

@Article{Wang2019a,
  author    = {Wang, Peng and Wu, Houwen and Yang, Haitang and Ying, Shuxuan},
  journal   = {Journal of High Energy Physics},
  title     = {Non-singular string cosmology via $\alpha$' corrections},
  year      = {2019},
  issn      = {1029-8479},
  month     = oct,
  number    = {10},
  volume    = {2019},
  doi       = {10.1007/jhep10(2019)263},
  publisher = {Springer Science and Business Media LLC},
}

\end{document}